\documentclass[11pt, reqno]{amsart}

\usepackage[a4paper,
            top=30mm,bottom=30mm,
            left=23mm,right=23mm]{geometry}

\usepackage{amsmath, amsthm, amsfonts, amssymb, amscd, mathtools, bm} 
\usepackage{xcolor} 
\usepackage[shortlabels]{enumitem} 
\usepackage[T1]{fontenc}
\usepackage{empheq}
\usepackage{placeins}
\usepackage[unicode=true]{hyperref}

\theoremstyle{plain}
\newtheorem{theorem}{Theorem}[section]
\newtheorem{lemma}[theorem]{Lemma}

\newtheorem{proposition}[theorem]{Proposition}

\theoremstyle{definition}
\newtheorem{definition}[theorem]{Definition}

\theoremstyle{remark}
\newtheorem{remark}[theorem]{Remark}

\numberwithin{equation}{section}

\newcommand{\rd}{\mathrm{d}}
\newcommand{\dx}{\mathrm{d}x}

\newcommand{\X}{\mathrm{X}}
\newcommand{\Z}{\mathrm{Z}}
\newcommand{\N}{\mathbb{N}}
\newcommand{\R}{\mathbb{R}}
\newcommand{\C}{\mathbb{C}}

\newcommand{\del}{\partial}
\newcommand{\eps}{\varepsilon}
\newcommand{\idot}{\!\cdot\!}

\title[]{Local analytic well-posedness for one-dimensional Vlasov--Dirac--Benney-type equations}

\author[]{Nuno J. Alves}
\author[]{Peter Markowich}
\author[]{Athanasios E. Tzavaras}

\address[N. J. Alves]{
      University of Vienna, Faculty of Mathematics, Oskar-Morgenstern-Platz 1, 1090 Vienna, Austria.}
\email{nuno.januario.alves@univie.ac.at}

\address[P. Markowich]{University of Vienna, Faculty of Mathematics, Oskar-Morgenstern-Platz 1, 1090 Vienna, Austria.}
\email{peter.markowich@univie.ac.at}

\address[A. E. Tzavaras]{King Abdullah University of Science and Technology, CEMSE Division, Thuwal, Saudi Arabia, 23955-6900.}
\email{athanasios.tzavaras@kaust.edu.sa}

\begin{document}

\begin{abstract}
We study a one-dimensional nonlinear Vlasov equation with a local self-consistent force field generated by the density, where the force is given by the spatial derivative of a real-analytic nonlinearity. For small analytic initial data, we prove local-in-time existence and uniqueness of analytic solutions. In particular, this yields a perturbative well-posedness result around the trivial equilibrium. We also give an energy-based representation of weak stationary states and discuss perturbations around spatially homogeneous stationary profiles. The proof relies on a contraction mapping argument in a complete metric space of analytic functions. As a technical byproduct, we establish quantitative composition estimates for analytic nonlinearities in the analytic norms used in the argument.
\end{abstract}


\maketitle
\thispagestyle{empty}

\section{Introduction}\label{sec:intro}

Kinetic equations of Vlasov type describe the evolution of a (nonnegative) phase-space distribution function
$f=f(t,x,v)$ under the action of a self-consistent force field. A convenient unifying framework
that incorporates several models of interest is provided by the Vlasov-type model
\begin{equation}\label{eq:intro_vlasov}
\partial_t f + v\cdot \nabla_x f - \nabla_x \frac{\delta \mathcal{E}}{\delta \rho}\cdot \nabla_v f = 0,
\qquad
\rho=\int f\,\rd v.
\end{equation}
The self-consistent equation~\eqref{eq:intro_vlasov} is posed for time $t \geq 0$ (although it is time-reversible) on the full phase space $\R_x^n \times \R_v^n$, where $x$ denotes position and $v$ denotes velocity. Here, $\mathcal{E}=\mathcal{E}(\rho)$ is an energy functional depending on the density, and  $\frac{\delta \mathcal{E}}{\delta \rho}$ 
stands for (the generator of) its  functional derivative. 
Different choices of $\mathcal{E}$ lead to different Vlasov-type models, notably the case of the Vlasov--Poisson system and the 
kinetic equation associated with Bohm potentials.

All these models are examples of infinite-dimensional Hamiltonian systems,  $\del_t f = [f, \mathcal{H}]$,
 generated via a Poisson bracket using the Hamiltonian functional
\begin{equation} \label{hamiltonian}
\mathcal{H}(f) = \iint \tfrac{1}{2} |v|^2 f  \, \rd v \, \rd x + \mathcal{E}(\rho), 
\end{equation}
that involves the kinetic energy and the potential energy $\mathcal{E}(\rho)$. This Hamiltonian structure was  pioneered by Morrison \cite{morrison1980,morrison1986}
for the Vlasov-Poisson system and the magnetohydrodynamics equations and by Dubrovin--Novikov \cite{DN1983, DN1984} for the equations of one-dimensional hydrodynamics.  
This structure is outlined in Section~\ref{section_structure} in the specific context of \eqref{eq:intro_vlasov}  and its connection for monokinetic
distributions to the Euler flows \eqref{eq:moment_system_general} introduced in \cite{GLT2017} for relative energy calculations.
The reader is referred to \cite{morrison1998} for a survey and detailed references on the subject of Hamiltonian description in fluid mechanics.

A most interesting paradigm of  \eqref{eq:intro_vlasov} arises when the energy is the Fisher information,
\begin{equation} \label{eq:intro_fisher}
\mathcal{E}=\mathcal{E}_{\rm F}(\rho)\coloneqq\int \tfrac12|\nabla\sqrt{\rho}|^2 \, \dx,
\end{equation}
in which case the variational derivative produces the Bohm potential: 
$
\frac{\delta \mathcal{E}_{\rm F}}{\delta\rho}(\rho)= -\frac{\Delta\sqrt{\rho}}{2\sqrt{\rho}}
$.
The induced force in \eqref{eq:intro_vlasov} involves higher derivatives and exhibits a strong sensitivity to the vacuum set $\{\rho = 0 \}$.
This Bohmian regime has been analysed, in particular, from the perspective of monokinetic distributions and phase-space (Bohmian) measures;
see \cite{markowich2010bohmian,markowich2012dynamics,figalli2014wkb}.

\par
A second example, closer to hydrodynamic models, is provided by internal energies
\begin{equation}\label{eq:intro_hydro}
\mathcal{E}=\mathcal{E}_{\rm H}(\rho)\coloneqq \int h(\rho) \, \dx \, ,
\end{equation}
where the internal energy density function $h$ is defined by
\begin{equation}\label{eq:intro_h_function}
h(\rho)=
\begin{dcases}
\tfrac{1}{\gamma-1} \rho^\gamma, & \gamma>1,
\vspace{2mm} \\
 \rho\log\rho, & \gamma=1.
\end{dcases}
\end{equation}
and results in a local force
\begin{equation} \label{eq:intro_forceH}
\nabla_x\frac{\delta\mathcal{E}_{\rm H}}{\delta\rho}= \nabla_x h'(\rho) =
\begin{dcases}
\gamma\rho^{\gamma-2} \nabla_x \rho,  & \gamma > 1, 
\vspace{2mm} \\
\nabla_x \log \rho , & \gamma=1.
\end{dcases}
\end{equation}
Here, the case $\gamma=2$ corresponds to the Vlasov--Dirac--Benney model;
see \cite{jabin2011analytic,bardos2012vlasov,bardos2013cauchy,bardos2015hamiltonian,han2016quasineutral} for well-posedness results and \cite{gibbonstsarev1999} for connections to the theory of water waves and conformal mappings. The singular borderline case $\gamma=1$ has been studied in \cite{carles2017monokinetic}, where existence of monokinetic solutions is proven as limits of the Wigner transform of solutions to a related logarithmic Schr\"{o}dinger equation.
\par At the hydrodynamic (Euler) level, the internal energy functional $\mathcal{E}_{\rm H}$ induces a barotropic pressure law via the identity
\begin{equation} \label{eq:intro_pressure_function}
p(\rho)=\rho h^\prime(\rho)-h(\rho),
\end{equation}
so that $p(\rho) = \rho^\gamma$ for $\gamma \geq 1$ when $h$ is given by \eqref{eq:intro_h_function}. Differentiating \eqref{eq:intro_pressure_function} yields
$p'(\rho)=\rho\,h''(\rho)$, and therefore
\begin{equation} \label{eq:intro_Boltzmann_relation}
\nabla_x p(\rho) = \rho \nabla_x \frac{\delta\mathcal{E}_{\rm H}}{\delta\rho},
\end{equation}
which for $h$ as in \eqref{eq:intro_h_function} corresponds to the Boltzmann relation of adiabatic electrons (see \cite{guo2011global}) when $\gamma=1$, whereas for $\gamma > 1$ it has been deduced from the zero-electron-mass limit of a bipolar Euler--Poisson system \cite{alves2024zero}. 
\par 
A third classical example is obtained when $\mathcal{E}$ represents the electrostatic energy. Formally, one takes
\begin{equation} \label{eq:intro_poisson}
\mathcal{E}=\mathcal{E}_{\rm P}(\rho)\coloneqq \int \tfrac12 |\nabla \phi|^2\,\dx,
\qquad
-\Delta \phi=\rho,
\end{equation}
in which case $\nabla_x \frac{\delta\mathcal{E}_{\rm P}}{\delta\rho}(\rho)=\nabla_x \phi$ and \eqref{eq:intro_vlasov} reduces 
to the Vlasov--Poisson system. The field $\phi$ is nonlocal and, due to the elliptic equation $-\Delta\phi=\rho$, 
the force $-\nabla_x\phi$ gains one derivative from $\rho$. This nonlocal smoothing structure underlies much of its well-posedness theory;
see \cite{bardos1985global,lions1991propagation,pfaffelmoser1992global,han2016quasineutral,ambrosio2017lagrangian}.

\subsection{Main results and strategy}

This work has two objectives. (i) To outline the infinite-dimensional Hamiltonian structure associated with equation \eqref{eq:intro_vlasov}. This part
is formal and is done in Section~\ref{section_structure}. (ii) To establish a local-in-time analytic well-posedness theory for the one-dimensional Vlasov model with local self-consistent field $-\partial_x\Phi(\rho)$, where $\Phi$ is real-analytic on $(-R,R)$ for some $R>0$. As a complement, in the final section we discuss the structure of stationary states and small analytic perturbations of spatially homogeneous profiles.
\par
The study of well-posedness for Vlasov-type models \eqref{eq:intro_vlasov} generated by energy functionals has to deal with the loss of derivatives 
and the problem of vacuum. The difficulty ranges from nonlocal Poisson on one side where a gain of derivatives occurs, to the Bohmian and logarithmic forces on the other extreme, where in addition to the loss of derivatives 
vacuum poses additional singularity issues.
Certain intermediate cases are already well understood, including the classical Vlasov--Poisson system and the Vlasov--Dirac--Benney model. In this paper, as a first step in this program, we focus on the intermediate regime of local singular force fields within a framework of real-analytic functions in one space dimension, for the kinetic model
\begin{equation}\label{vlasov0}
\begin{dcases}
\partial_t f + v\cdot \partial_x f -\partial_x \Phi(\rho) \cdot \partial_v f = 0, \\
\rho = \int f \, \rd v, \\
f|_{t=0} = f_0,
\end{dcases}
\end{equation}
for $(t,x,v) \in [0,T] \times \R^2$, where $\Phi$ is \emph{real-analytic} on $(-R,R)$ in the sense that it admits a globally convergent power series expansion (see Definition~\ref{def:entire}),
\begin{equation} \label{eq:intro_Phi}
\Phi(z)=\sum_{n\ge 0} c_n z^n \qquad (|z| < R).
\end{equation}

This assumption covers a class of local force laws --- for instance, polynomial laws $\Phi(\rho) = \rho^{q+1}$, where  $q \ge 0 $  is an integer, or exponential nonlinearities --- while keeping the nonlinearity compatible with the analytic framework. In the Hamiltonian/energy-driven viewpoint
\eqref{eq:intro_vlasov}, the function $\Phi$ corresponds to the variational derivative
$\delta\mathcal E/\delta\rho$; in particular, in the internal energy case~\eqref{eq:intro_hydro}
one has $\delta\mathcal E/\delta\rho=h'(\rho)$, so $\Phi$ plays the role of the enthalpy function $h'$.
\par  The employed analytic norms enjoy an algebra property 
(Lemmas~\ref{lemma_algebra}--\ref{lemma_algebraH}) and, when combined with the majorant series $\tilde\Phi,\tilde\Phi',\tilde\Phi''$, they provide a robust calculus for controlling repeated Leibniz and chain rules in the composition $x\mapsto \Phi(\rho(x))$. By contrast, singular choices such as 
$\Phi(\rho)=\log\rho$ or forces driven by the Bohm potential fall outside this class. Additional ideas are then required 
to control the analysis near vacuum; we do not address these cases here.
\par 
Our main theorem establishes local-in-time existence and uniqueness of analytic solutions to~\eqref{vlasov0} in a suitable complete metric space, under explicit smallness conditions on the initial datum; see Theorem~\ref{mainthm}. The proof is based on a fixed-point argument applied to a contraction mapping on a certain metric space of analytic functions; see the space \eqref{spaceX} with the norm \eqref{normZ}. A key technical ingredient is a time-dependent analyticity radius $\lambda(t)$. Allowing the radius to decrease in time provides the extra room needed in the analytic estimates to balance the loss of one spatial derivative coming from the local force $-\partial_x\Phi(\rho)$. We emphasize that, unlike a direct application of the Cauchy--Kowalevski theorem (see, \emph{e.g.}\ \cite{folland1995pde}), our proof gives quantitative control in the norms defining \eqref{spaceX} and in \eqref{normZ}, and obtains existence and uniqueness by a contraction argument.
\par 

From the modelling viewpoint, only initial data $f_0\ge 0$ are physically relevant, in which case~$f$ is interpreted as a density in phase space.
For smooth solutions of \eqref{vlasov0}, the equation is a transport equation and $f$ is constant along characteristics; hence the sign of $f$ is preserved, and $f_0\ge 0$ implies $f(t)\ge 0$ on its whole domain of existence. Our analytic fixed-point argument, however, does not require any sign condition on the initial datum.

\subsection{Relation to the literature and contributions}

The analytic well-posedness strategy we use goes back to Jabin and Nouri~\cite{jabin2011analytic} in the study of the
Vlasov--Dirac--Benney case $\Phi(\rho)=\rho$, and was also employed in~\cite{bossy2013local} for a related Vlasov-type
model with stochastic interpretation. More broadly, analytic frameworks for Vlasov equations play a central role in the
work of Mouhot and Villani~\cite{mouhot2011landau} on nonlinear Landau damping and in subsequent developments;
see, for instance, \cite{gagnebin2023landau,hankwan2025linear,nguyen2025remarks}. Related analytic techniques also appear
in the Gevrey setting \cite{bedrossian2016landau, grenier2021landau, ruiz2021gevrey}, which lies outside the framework
considered here. For background on real-analytic functions we refer to~\cite{krantz2002primer}.
\par

The contributions of this paper are threefold.
First, motivated by Morrison's Hamiltonian formalism for the Vlasov--Poisson system \cite{morrison1980,morrison1986},
we extend the Poisson-bracket framework to Hamiltonians of the form \eqref{hamiltonian} with density-dependent potential energy $\mathcal E(\rho)$,
thereby recovering the generalized Vlasov model \eqref{eq:intro_vlasov} and its formal monokinetic reduction to Euler-type flows.
Second, we establish a quantitative local-in-time analytic well-posedness theory for the one-dimensional Vlasov model with local self-consistent field
$-\partial_x\Phi(\rho)$ for general real-analytic nonlinearities $\Phi$.
Compared with~\cite{jabin2011analytic}, we treat a nonlinear dependence $\Phi(\rho)$ beyond the linear case and provide explicit analytic estimates
for the composition $x\mapsto \Phi(\rho(x))$ in terms of the majorant series $\tilde\Phi,\tilde\Phi',\tilde\Phi''$, within a complete time-dependent analytic framework.
The main theorem is formulated as a small-data result around the trivial stationary solution $f\equiv 0$.
Third, in the last section we show that weak stationary states of \eqref{vlasov0} admit an energy-based representation, which strongly constrains their structure. This, in turn, allows us to extend the fixed-point argument to small analytic perturbations of spatially homogeneous stationary profiles.

%

\subsection{Organization of the paper}
Section~\ref{section_structure} provides a formal modelling discussion based on a Hamiltonian formulation, motivating the energy-driven Vlasov structure \eqref{eq:intro_vlasov}. 
Section~\ref{section_analytic} introduces the analytic function spaces and their associated norms, and establishes the basic calculus needed for the nonlinear analysis, including algebra estimates and quantitative bounds for composition with real-analytic functions. Section~\ref{section_main} states the main theorem and proves it by a contraction mapping argument in a suitable time-dependent analytic framework. Finally, Section~\ref{sec:perturb_nontrivial} describes stationary states in terms of the local energy
and discusses perturbations around nontrivial stationary solutions, in particular $x$-independent profiles.
 
\section{Generalized Vlasov equations} \label{section_structure} 
In this section, we extend a Hamiltonian formalism developed for the Vlasov--Poisson system by Morrison \cite{morrison1980,morrison1986} to the 
Hamiltonian $\mathcal{H} (f)$ given by \eqref{hamiltonian},  depending on $f = f(x,v)$ through the kinetic energy and a density-dependent 
potential energy functional $\mathcal{E}(\rho)$. Within this framework we derive the Vlasov model \eqref{eq:intro_vlasov} from the Hamiltonian functional \eqref{hamiltonian}. We also briefly discuss how, under a monokinetic ansatz, the resulting kinetic dynamics formally reduces to Euler flows.
This formal analysis gives rise to an infinite-dimensional Hamiltonian system. Its structure is well-understood 
within the community studying Hamiltonian systems, \emph{e.g.}~\cite{morrison1986,morrison1998}; we 
outline the key steps here for the reader's convenience.

\subsection{Hamiltonian flow}  \label{section_hamiltonian} 
Let $f = f(t,x,v)$ be a phase-space distribution function, evolving in time $t>0$ for $x,v \in \R^d$, with $d \in \N$ being the dimension. 
In this section we indicate how the equation \eqref{eq:intro_vlasov} may be formally interpreted as an infinite-dimensional Hamiltonian
system,
\begin{equation}
\del_t f = [f, \mathcal{H}]
\end{equation}
emerging from the Hamiltonian functional
$$
\mathcal{H}(f) = \iint \tfrac{1}{2} |v|^2 f  \, \rd v \, \rd x + \mathcal{E}(\rho), 
$$
where $\rho = \int f \, \rd v$ denotes the zeroth velocity moment of $f$.

This is achieved by adapting the framework discovered in \cite{morrison1986} for the Vlasov--Poisson system to the present 
generalized Vlasov model.
We work formally and assume that boundary terms vanish (for instance, $x$ periodic and $f$ rapidly decaying in $v$), and
often omit the explicit dependence on $t$.
The Poisson bracket $[\cdot, \cdot]$ between two functionals $A = A(f)$ and $B = B(f)$ is defined by
\begin{equation}\label{poissonbracket}
[ A, B] = \iint f(x', v') \left \{ \frac{\delta A}{\delta f} , \frac{\delta B}{\delta f} \right\} (x', v')  \, \rd x' \, \rd v' \, ,
\end{equation}
where $\delta A /\delta f$, $\delta B /\delta f$ are the generators of functional derivatives of $A$, $B$, while $\{ \cdot, \cdot \}$ stands for the canonical 
Poisson bracket defined for two functions $F(x,v)$, $G(x,v)$ by
\begin{equation*}
\{F,G \} = \nabla_x F \cdot \nabla_v G- \nabla_v F \cdot \nabla_x G .
\end{equation*}
As the canonical Poisson bracket $\{\cdot, \cdot \}$ satisfies the Jacobi identity this property is inherited by the (functional) bracket $[\cdot, \cdot]$.

\par
Choosing $A(f)$ to be the point-evaluation functional, $A(f):=f(x,v)$, and $B(f)$ to be the Hamiltonian functional, a formal calculation outlined 
below yields
\begin{equation}\label{evolutionH}
\partial_t f=[f,\mathcal{H}]=-\left\{f,\frac{\delta\mathcal{H}}{\delta f}\right\},
\end{equation}
where we note that the first Poisson bracket is a bracket of functionals whereas the second is a canonical Poisson bracket of functions.
Indeed, for $A(f)= f  (x,v,t)$ the point evaluation functional we have 
\[
\frac{\delta A}{\delta f}(x',v')=\delta(x-x')\,\delta(v-v'),
\]
while for the Hamiltonian functional $\mathcal{H}$,
$$
\Big \langle \frac{\delta \mathcal{H}}{\delta f} , F \Big \rangle 
= \frac{\rd}{\rd\epsilon} \Big |_{\epsilon = 0} \mathcal{H}(f + \epsilon F)
= \iint \tfrac{1}{2} |v|^2  F \, \rd x \, \rd v + \Big\langle \frac{\delta \mathcal{E}}{\delta \rho} (\rho), \int F \, \rd v \Big\rangle.
$$
Inserting these formulas into \eqref{poissonbracket} gives
\begin{align*}
\big[ f,\mathcal{H} \big] &=  \; \iint f(x' , v') \Big [ 
\nabla_{x'} \Big ( \delta (x- x') \delta (v - v') \Big ) \cdot \nabla_{v'} \Big ( \frac{\delta \mathcal{H}}{\delta f} \Big )
\\
&  \qquad \qquad \qquad \qquad \qquad - \nabla_{v'} \Big ( \delta (x- x') \delta (v - v') \Big ) \cdot \nabla_{x'} \Big ( \frac{\delta \mathcal{H}}{\delta f} \Big ) \Big ] 
 \rd x' \, \rd v'
\\
&= - \iint   \delta (x- x') \delta (v - v')  \Big [ \nabla_{x'}   f (x' , v') \cdot \nabla_{v'} \Big (   \frac{\delta \mathcal{H}}{\delta f}  \Big )
  - \nabla_{v'} f (x' , v') \cdot \nabla_{x'} \Big (  \frac{\delta \mathcal{H}}{\delta f}  \Big ) \Big ] \rd x' \, \rd v'
\\
&= - \Big \{ f , \frac{\delta \mathcal{H}}{\delta f} \Big \} (x,v) 
\\
&= -\Big( v\cdot\nabla_x f - \nabla_x\frac{\delta\mathcal E}{\delta\rho}(\rho)\cdot \nabla_v f\Big).
\end{align*}
Therefore \eqref{evolutionH} is exactly \eqref{eq:intro_vlasov}.

\medskip
\subsection{Monokinetic distributions and Euler systems}
In addition to the zeroth velocity moment
\[
\rho=\int f\,\rd v,
\]
we introduce the first velocity moment
\[
J=\rho u=\int vf\,\rd v.
\]
\par 
A formal computation based on \eqref{eq:intro_vlasov} yields the zero-th and first order moment system
\begin{equation}\label{eq:moment_system_general}
\begin{cases}
\partial_t \rho+\nabla_x\cdot(\rho u)=0,
\\[1mm]
\partial_t(\rho u)+\nabla_x\cdot(\rho u\otimes u)+\nabla_x\cdot P
= -\,\rho\,\nabla_x\dfrac{\delta\mathcal E}{\delta\rho},
\end{cases}
\end{equation}
where the (kinetic) stress tensor $P$ is given by
\begin{equation}\label{eq:stress_tensor}
P=\displaystyle\int (v-u)\otimes (v-u)\,f\,\rd v.
\end{equation}
The moment system \eqref{eq:moment_system_general} is not closed in general, since $P$ depends on the distribution $f$
and requires the second-order moments to be computed.

\par 
Assume now that $f$ is monokinetic, namely
\begin{equation}\label{eq:monokinetic_single}
f(t,x,v)=\rho(t,x)\,\delta\big(v-u(t,x)\big).
\end{equation}
Then $P\equiv 0$ in \eqref{eq:stress_tensor}, and \eqref{eq:moment_system_general} reduces to the compressible Euler-type system
\begin{equation}\label{eq:euler_general_energy}
\begin{cases}
\partial_t \rho+\nabla_x\cdot(\rho u)=0,
\\[1mm]
\partial_t(\rho u)+\nabla_x\cdot(\rho u\otimes u)
= -\,\rho\,\nabla_x\dfrac{\delta\mathcal E}{\delta\rho}.
\end{cases}
\end{equation}
System~\eqref{eq:euler_general_energy} was introduced in \cite{GLT2017} as a
paradigm to develop relative entropy properties for Hamiltonian systems.
Different choices of the energy functional recover several classical fluid models: taking $\mathcal{E}=\mathcal{E}_{\rm F}$~ \eqref{eq:intro_fisher} yields the (quantum) Euler system driven by the Bohm potential, taking $\mathcal{E}=\mathcal{E}_{\rm H}$ ~\eqref{eq:intro_hydro} yields the barotropic compressible Euler system, and taking $\mathcal{E}=\mathcal{E}_{\rm P}$~\eqref{eq:intro_poisson} yields the Euler--Poisson system.

\section{Analytic functions and associated norms} \label{section_analytic}
We start by fixing some notation that will be used throughout the rest of the manuscript.
The sets of positive and nonnegative integers are denoted by $\mathbb{N}$ and $\mathbb{N}_0$,
respectively. For functions depending on $x$, $v$, or both, the corresponding supremum norm,
$\| \idot \|_\infty$, refers to the supremum taken over $x$, $v$, or both, respectively.
Derivatives with respect to $x$ (resp.\ $v$) are denoted by the partial derivative $\partial_x$
(resp.\ $\partial_v$). Derivatives with respect to $t$ (resp.\ $\lambda$) are denoted by the total
derivative $\rd/\rd t$ (resp.\ $\rd/\rd\lambda$).

\subsection{Analytic functions and norms}
We shall deal with analytic functions on $\R^2$ (resp.\ on $\R$) which satisfy the following uniform
$L^\infty$ derivative bounds: there exist constants $C>0$ and $\Lambda>0$ such that
\begin{equation}\label{analyticdef}
 \frac{\Lambda^{k+\ell}}{k!\,\ell!}\,\|\partial_x^k \partial_v^\ell f \|_\infty \le C
 \qquad \forall\, k,\ell \in \N_0 .
\end{equation}
When \eqref{analyticdef} holds we say that $f$ is analytic with (uniform) radius of convergence $\Lambda$.
\par
If $f$ only depends on $x$ (resp.\ $v$), then it is analytic if \eqref{analyticdef} holds for $\ell=0$
(resp.\ $k=0$). The constant $\Lambda$ is called the radius of convergence since, for each
$(y,w) \in \mathbb{R}^2$, condition \eqref{analyticdef} implies the bound
\[
\sum_{k,\ell \geq 0} \frac{|\partial_x^k \partial_v^\ell f(y,w)|}{k!\ell!}\,|x-y|^k\,|v-w|^\ell
\leq C \sum_{k \geq 0} \left( \frac{|x-y|}{\Lambda} \right)^k
      \sum_{\ell \geq 0} \left( \frac{|v-w|}{\Lambda} \right)^\ell,
\]
which converges whenever $|x-y| < \Lambda$ and $|v-w| < \Lambda$.\par

Now let $\lambda > 0$. Within this analytic functions framework, the first norm that we consider is
the following:
\begin{equation} \label{normlambda}
\| f \|_\lambda = \sum_{k,\ell \geq 0} \frac{\lambda^{k+\ell}}{k!\ell!}\| \partial_x^k \partial_v^\ell f \|_\infty.
\end{equation}
One can readily check that if $f$ is analytic with radius of convergence $\Lambda$, then
$\| f\|_\lambda < \infty$ for every $\lambda < \Lambda$. Moreover, for each $n \in \mathbb{N}_0$ we
consider the $n^{\text{th}}$ order derivative of $\|\!\cdot\!\|_\lambda$ with respect to $\lambda$:
\begin{equation} \label{dnormlambda}
| f |_{\lambda,n} = \frac{\rd^n}{\rd\lambda^n}\| f \|_\lambda
= \sum_{k+\ell \geq n} \frac{(k+\ell)!}{(k+\ell-n)!} \frac{\lambda^{k+\ell-n}}{k!\ell!}\| \partial_x^k \partial_v^\ell f \|_\infty.
\end{equation}
Furthermore, we consider a norm $\|\!\cdot\!\|_{H,\lambda}$ and a seminorm $|\!\cdot\!|_{H, \lambda}$
given by
\begin{equation} \label{normsH}
 \| f \|_{H,\lambda} = \sum_{n\geq 0} \frac{1}{(n!)^2} |f|_{\lambda,n}
 \qquad \text{and} \qquad
 | f |_{H, \lambda}= \sum_{n\geq 1} \frac{n^2}{(n!)^2} |f|_{\lambda,n}.
\end{equation}

The next lemma follows from \cite[Lemma~2.1 and Lemma~2.3(i)]{bossy2013local}.
\begin{lemma} \label{lemma1}
Let $f$ be analytic with radius of convergence $\Lambda$. Then, for each $0 < \lambda < \Lambda$,
\begin{equation} \label{eq:analytic_finite_H}
\| f \|_{H,\lambda} < \infty, \qquad
| f |_{H, \lambda} < \infty,
\end{equation}
and, for $n \in \mathbb{N}_0$,
\begin{equation}\label{eq:dx_shift}
|\partial_x f |_{\lambda, n} \leq |f|_{\lambda, n+1}.
\end{equation}
\end{lemma}

\subsection{Algebra norm estimates}

\begin{lemma}\label{lemma_algebra}
Let $f,g$ be analytic with radius of convergence $\Lambda$, and let $0<\lambda<\Lambda$.
Then, for every $n\in\N_0$,
\begin{equation}\label{eq:algebra_conv}
|fg|_{\lambda,n}\le \sum_{m=0}^n {n\choose m}\,|f|_{\lambda,m}\,|g|_{\lambda,n-m}.
\end{equation}
In particular, for $n=0$,
\begin{equation}\label{eq:algebra}
\|fg\|_\lambda\le \|f\|_\lambda\,\|g\|_\lambda,
\end{equation}
and for $n=1$,
\begin{equation}\label{eq:algebra_dlambda}
|fg|_{\lambda,1}\le |f|_{\lambda,1}\,\|g\|_\lambda+\|f\|_\lambda\,|g|_{\lambda,1}.
\end{equation}
\end{lemma}

\begin{proof}
For $h$ analytic, set
\[
A_{k,\ell}(h)=\frac{1}{k!\,\ell!}\,\|\partial_x^k\partial_v^\ell h\|_\infty\ge 0,
\qquad
\|h\|_\lambda=\sum_{k,\ell\ge 0}A_{k,\ell}(h)\,\lambda^{k+\ell}.
\]
By the Leibniz rule, for all $k,\ell\in\N_0$,
\[
\partial_x^k\partial_v^\ell(fg)=\sum_{i=0}^k\sum_{j=0}^\ell
\binom{k}{i}\binom{\ell}{j}\,\partial_x^i\partial_v^j f\cdot \partial_x^{k-i}\partial_v^{\ell-j} g,
\]
hence, after taking $\|\idot\|_\infty$ and dividing by $k!\,\ell!$,
\[
A_{k,\ell}(fg)\le \sum_{i=0}^k\sum_{j=0}^\ell A_{i,j}(f)\,A_{k-i,\ell-j}(g).
\]

Fix $n\in\N_0$. Differentiating termwise in $\lambda$ gives
\[
|h|_{\lambda,n}=\frac{\rd^n}{\rd\lambda^n}\|h\|_\lambda
=\sum_{k,\ell\ge 0}A_{k,\ell}(h)\,(k+\ell)_n\,\lambda^{k+\ell-n},
\]
where $(p)_n:=p(p-1)\cdots(p-n+1)$ with the convention $(p)_n=0$ if $p<n$. Multiplying the componentwise inequality by
$(k+\ell)_n\,\lambda^{k+\ell-n}$ and summing over $k,\ell\ge 0$, we obtain
\[
|fg|_{\lambda,n}
\le \sum_{k,\ell\ge 0}\sum_{i=0}^k\sum_{j=0}^\ell
A_{i,j}(f)\,A_{k-i,\ell-j}(g)\,(k+\ell)_n\,\lambda^{k+\ell-n}.
\]
Reindex with $a=i$, $b=j$, $c=k-i$, $d=\ell-j$ to get
\[
|fg|_{\lambda,n}
\le \sum_{a,b,c,d\ge 0} A_{a,b}(f)\,A_{c,d}(g)\,\big((a+b)+(c+d)\big)_n\,\lambda^{(a+b)+(c+d)-n}.
\]
Using the identity
\[
(p+q)_n=\sum_{m=0}^n\binom{n}{m}\,(p)_m\,(q)_{n-m},
\]
we infer
\begin{align*}
|fg|_{\lambda,n}
&\le \sum_{m=0}^n \binom{n}{m}
\Big(\sum_{a,b\ge 0}A_{a,b}(f)\,(a+b)_m\,\lambda^{a+b-m}\Big)
\Big(\sum_{c,d\ge 0}A_{c,d}(g)\,(c+d)_{n-m}\,\lambda^{c+d-(n-m)}\Big) \\
&= \sum_{m=0}^n \binom{n}{m}\,|f|_{\lambda,m}\,|g|_{\lambda,n-m},
\end{align*}
which is \eqref{eq:algebra_conv}.
\end{proof}

\begin{lemma}\label{lemma_algebraH}
Let $f,g$ be analytic with radius of convergence $\Lambda$, and let $0 < \lambda < \Lambda$. Then
\begin{equation}\label{eq:algebraH}
\| f g \|_{H,\lambda} \leq \| f \|_{H,\lambda} \| g \|_{H,\lambda},
\end{equation}
and
\begin{equation}\label{eq:algebra_dlambdaH}
| f g |_{H,\lambda} \leq 2| f |_{H,\lambda} \| g \|_{H,\lambda} + 2\| f \|_{H,\lambda} | g |_{H,\lambda}.
\end{equation}
\end{lemma}
\begin{proof}
Using \eqref{eq:algebra_conv} and multiplying by $\frac{1}{(n!)^2}$ gives
\begin{align*}
\frac{1}{(n!)^2}|fg|_{\lambda,n}
& \le \sum_{m=0}^n \frac{1}{(n!)^2}{n\choose m}\,|f|_{\lambda,m}\,|g|_{\lambda,n-m}\\ 
& \leq \sum_{m=0}^n \frac{|f|_{\lambda,m}}{(m!)^2}\,\frac{|g|_{\lambda,n-m}}{((n-m)!)^2},
\end{align*}
since
\[
\frac{1}{(n!)^2}{n\choose m}
=\frac{1}{n!}\,\frac{1}{m!(n-m)!}
\le \frac{1}{(m!)^2}\,\frac{1}{((n-m)!)^2}.
\]
Summing over $n\ge 0$ and reindexing with $r=n-m$ yields
\[
\|fg\|_{H,\lambda}
=\sum_{n\ge 0}\frac{1}{(n!)^2}|fg|_{\lambda,n}
\le \sum_{m\ge 0}\frac{|f|_{\lambda,m}}{(m!)^2}\sum_{r\ge 0}\frac{|g|_{\lambda,r}}{(r!)^2}
=\|f\|_{H,\lambda}\,\|g\|_{H,\lambda},
\]
which is \eqref{eq:algebraH}.
\par
Now, we combine the inequality
\begin{equation}\label{eq:n2_split}
n^2=(m+(n-m))^2\le 2m^2+2(n-m)^2
\qquad\text{for all }0\le m\le n,
\end{equation}
with the same factorial bound as above to obtain
\[
\frac{n^2}{(n!)^2}{n\choose m}
\le \frac{2m^2}{(m!)^2}\frac{1}{((n-m)!)^2}
+\frac{2}{(m!)^2}\frac{(n-m)^2}{((n-m)!)^2}.
\]
Therefore, using \eqref{eq:algebra_conv},
\begin{align*}
|fg|_{H,\lambda}
= & \ \sum_{n\ge 1}\frac{n^2}{(n!)^2}|fg|_{\lambda,n}\\
\leq & \ \sum_{n\ge 1} \sum_{m=0}^n \frac{n^2}{(n!)^2}{n\choose m} |f|_{\lambda,m} |g|_{\lambda,n-m}\\
\le & \ 2\sum_{n\ge 1}\sum_{m=0}^n \frac{m^2}{(m!)^2}|f|_{\lambda,m}\,\frac{1}{((n-m)!)^2}|g|_{\lambda,n-m} \\
&+2\sum_{n\ge 1}\sum_{m=0}^n \frac{1}{(m!)^2}|f|_{\lambda,m}\,\frac{(n-m)^2}{((n-m)!)^2}|g|_{\lambda,n-m}.
\end{align*}
Reindexing each double sum with $r=n-m$ leads to
\begin{align*}
|fg|_{H,\lambda}
&\le 2\sum_{m\ge 1}\frac{m^2}{(m!)^2}|f|_{\lambda,m}\sum_{r\ge 0}\frac{1}{(r!)^2}|g|_{\lambda,r}
+2\sum_{m\ge 0}\frac{1}{(m!)^2}|f|_{\lambda,m}\sum_{r\ge 1}\frac{r^2}{(r!)^2}|g|_{\lambda,r}\\
&=2\,|f|_{H,\lambda}\,\|g\|_{H,\lambda}+2\,\|f\|_{H,\lambda}\,|g|_{H,\lambda},
\end{align*}
which is \eqref{eq:algebra_dlambdaH}.
\end{proof}

\medskip
\subsection{Real-analytic nonlinearities}

In addition to analytic functions of $(x,v)$ in the sense of~\eqref{analyticdef},
we shall use scalar nonlinearities acting by composition.

\begin{definition}\label{def:entire}
A function $\Phi:\R\to\R$ is said to be \emph{real-analytic on $(-R,R)$} if there exists $R>0$ and
real coefficients $(c_n)_{n\ge0}$ such that
\begin{equation}\label{eq:Phi_entire_series}
\Phi(z)=\sum_{n=0}^\infty c_n z^n \qquad \text{for all } |z|<R.
\end{equation}
For $0\le r<R$ we define the associated \emph{majorant functions}
\begin{equation}\label{eq:Phi_majorants}
\tilde{\Phi}(r)=\sum_{n=0}^\infty |c_n|\,r^n,
\qquad
\tilde{\Phi}'(r)=\sum_{n=1}^\infty n\,|c_n|\,r^{n-1}, \qquad
\tilde{\Phi}''(r)=\sum_{n=2}^\infty n(n-1)\,|c_n|\,r^{n-2}.
\end{equation}
\end{definition}

\begin{remark} \label{rem:majorant}
Let $\Phi$ be real-analytic on $(-R,R)$ with expansion $\Phi(z)=\sum_{n\ge0}c_n z^n$ for all $|z|<R$. Then, for every $r\in[0,R)$, the series defining $\tilde\Phi(r)$, $\tilde\Phi'(r)$ and $\tilde\Phi''(r)$ in
\eqref{eq:Phi_majorants} converge, hence these quantities are finite.
Indeed, by the Cauchy--Hadamard formula for the radius of convergence of a power series,
\[
\frac{1}{R}=\limsup_{n\to\infty}|c_n|^{1/n}.
\]
Fix $r\in[0,R)$. Then
\[
\limsup_{n\to\infty}\bigl(|c_n|\,r^n\bigr)^{1/n}
=
r\limsup_{n\to\infty}|c_n|^{1/n}
=
\frac{r}{R}<1,
\]
so the root test implies that $\sum_{n\ge0}|c_n|\,r^n$ converges. Hence $\tilde\Phi(r)<\infty$.
For $\tilde\Phi'$ and $\tilde\Phi''$ we use the same argument, noting that
\[
\lim_{n\to\infty}n^{1/n}=1
\qquad\text{and}\qquad
\lim_{n\to\infty}\bigl(n(n-1)\bigr)^{1/n}=1.
\]
Therefore,
\[
\limsup_{n\to\infty}\bigl(n|c_n|\,r^{n-1}\bigr)^{1/n}
=
\limsup_{n\to\infty}\Bigl(n^{1/n}\,|c_n|^{1/n}\,r^{(n-1)/n}\Bigr)
=
\frac{r}{R}<1,
\]
and similarly
\[
\limsup_{n\to\infty}\bigl(n(n-1)|c_n|\,r^{n-2}\bigr)^{1/n}=\frac{r}{R}<1.
\]
Hence the root test yields convergence of the series in \eqref{eq:Phi_majorants} defining
$\tilde\Phi'(r)$ and $\tilde\Phi''(r)$ as well. 
\end{remark}

\begin{lemma}\label{lemma_composition}
Let $\Phi$ be real-analytic on $(-R,R)$ with expansion $\Phi(z)=\sum_{n\ge0}c_n z^n$.
Let $u$ be analytic with radius of convergence $\Lambda$, and let $0<\lambda<\Lambda$ satisfy $\|u\|_\lambda<R$.
Then the series $\Phi(u)\coloneqq\sum_{n\ge0}c_n u^n$ converges in $\|\cdot\|_\lambda$. Moreover,
\begin{equation}\label{composition_est}
\|\Phi(u)\|_{\lambda} \leq \tilde{\Phi}(\|u\|_{\lambda}),
\end{equation}
and
\begin{equation}\label{composition_diff_est}
|\Phi(u)|_{\lambda, 1} \leq \tilde{\Phi}'(\|u\|_{\lambda}) \, |u|_{\lambda, 1}.
\end{equation}
\end{lemma}

\begin{proof}
By \eqref{eq:algebra}, $\|u^n\|_\lambda\le \|u\|_\lambda^n$ for every $n\in\N$. Hence
\[
\sum_{n=0}^\infty \|c_n u^n\|_\lambda
=\sum_{n=0}^\infty |c_n|\,\|u^n\|_\lambda
\le \sum_{n=0}^\infty |c_n|\,\|u\|_\lambda^n
= \tilde{\Phi}(\|u\|_\lambda)<\infty,
\]
so $\sum_{n\ge 0} c_n u^n$ converges in $\|\!\cdot\!\|_\lambda$ and defines $\Phi(u)$.
Moreover, by the triangle inequality and the previous estimate,
\[
\|\Phi(u)\|_\lambda
\le \sum_{n=0}^\infty \|c_n u^n\|_\lambda
\le \tilde{\Phi}(\|u\|_\lambda),
\]
which proves \eqref{composition_est}.

For \eqref{composition_diff_est}, we first claim that for every $n\ge 1$,
\begin{equation}\label{eq:power_dlambda}
|u^n|_{\lambda,1}\le n\,\|u\|_\lambda^{n-1}\,|u|_{\lambda,1}.
\end{equation}
Indeed, by \eqref{eq:algebra_dlambda} and induction,
\begin{align*}
|u^n|_{\lambda,1}& =|u^{n-1}u|_{\lambda,1}
 \\ 
 & \le |u^{n-1}|_{\lambda,1}\,\|u\|_\lambda+\|u^{n-1}\|_\lambda\,|u|_{\lambda,1} \\
& \le (n-1)\|u\|_\lambda^{n-1}|u|_{\lambda,1}+\|u\|_\lambda^{n-1}|u|_{\lambda,1},
\end{align*}
which yields \eqref{eq:power_dlambda}. Consequently,
\[
|\Phi(u)|_{\lambda,1}
\le \sum_{n=1}^\infty |c_n|\,|u^n|_{\lambda,1}
\le \sum_{n=1}^\infty n\,|c_n|\,\|u\|_\lambda^{n-1}\,|u|_{\lambda,1}
= \tilde{\Phi}'(\|u\|_\lambda)\,|u|_{\lambda,1},
\]
as desired.
\end{proof}

\begin{lemma}\label{lem:Phi_Lipschitz}
Let $\Phi$ be real-analytic on $(-R,R)$ with expansion $\Phi(z)=\sum_{n\ge0}c_n z^n$.
Let $u,v$ be analytic with radius of convergence $\Lambda$, and let $0<\lambda<\Lambda$.
Set $R_0=\max\{\|u\|_\lambda,\|v\|_\lambda\}$ and assume $R_0<R$.
Then
\begin{equation} \label{eq:Phi_Lipschitz}
\|\Phi(u)-\Phi(v)\|_\lambda
\le \tilde{\Phi}'(R_0)\,\|u-v\|_\lambda.
\end{equation}
\end{lemma}
\begin{proof}
For $n\ge 1$ we use the factorization
\[
u^n-v^n=(u-v)\sum_{m=0}^{n-1}u^m v^{n-1-m}.
\]
Using \eqref{eq:algebra} we estimate
\[
\|u^n-v^n\|_\lambda
\le \|u-v\|_\lambda \sum_{m=0}^{n-1}\|u^m v^{n-1-m}\|_\lambda
\le \|u-v\|_\lambda \sum_{m=0}^{n-1}\|u\|_\lambda^m \|v\|_\lambda^{n-1-m}.
\]
With $R_0=\max\{\|u\|_\lambda,\|v\|_\lambda\}$ this gives
\[
\|u^n-v^n\|_\lambda \le n\,R_0^{n-1}\,\|u-v\|_\lambda.
\]
Therefore,
\begin{align*}
\|\Phi(u)-\Phi(v)\|_\lambda
&\le \sum_{n=1}^\infty |c_n|\,\|u^n-v^n\|_\lambda
\le \sum_{n=1}^\infty n\,|c_n|\,R_0^{n-1}\,\|u-v\|_\lambda \\
&= \tilde{\Phi}'(R_0)\,\|u-v\|_\lambda,
\end{align*}
as claimed.
\end{proof}

\begin{lemma}\label{lem:Phi_diff_dlambda}
Let $\Phi$ be real-analytic on $(-R,R)$ with expansion $\Phi(z)=\sum_{n\ge0}c_n z^n$.
Let $u,v$ be analytic with radius of convergence $\Lambda$, and let $0<\lambda<\Lambda$.
Set $R_0=\max\{\|u\|_\lambda,\|v\|_\lambda\}$ and assume $R_0<R$.
Then
\begin{equation}\label{eq:Phi_diff_dlambda}
|\Phi(u)-\Phi(v)|_{\lambda,1}
\le \tilde{\Phi}'(R_0)\,|u-v|_{\lambda,1}
+\tilde{\Phi}''(R_0)\,(\,|u|_{\lambda,1}+|v|_{\lambda,1}\,)\,\|u-v\|_\lambda.
\end{equation}
\end{lemma}

\begin{proof}
Write
\[
\Phi(u)-\Phi(v)=\sum_{n=1}^\infty c_n\,(u^n-v^n),
\qquad
u^n-v^n=(u-v)\sum_{m=0}^{n-1}u^m v^{n-1-m}.
\]
By \eqref{eq:algebra_dlambda} and the triangle inequality,
\begin{equation}\label{eq:Phi_diff_start}
\begin{split}
|&\Phi(u) -\Phi(v)|_{\lambda,1} \\
&\le \sum_{n=1}^\infty |c_n|\,\Big|(u-v)\sum_{m=0}^{n-1}u^m v^{n-1-m}\Big|_{\lambda,1}\\
&\le \sum_{n=1}^\infty |c_n|\Big(
|u-v|_{\lambda,1}\,\Big\|\sum_{m=0}^{n-1}u^m v^{n-1-m}\Big\|_\lambda
+\|u-v\|_\lambda\,\Big|\sum_{m=0}^{n-1}u^m v^{n-1-m}\Big|_{\lambda,1}\Big).
\end{split}
\end{equation}
As in the proof of Lemma~\ref{lem:Phi_Lipschitz},
\[
\Big\|\sum_{m=0}^{n-1}u^m v^{n-1-m}\Big\|_\lambda
\le \sum_{m=0}^{n-1}\|u\|_\lambda^m\|v\|_\lambda^{n-1-m}
\le n\,R_0^{n-1},
\]
where $R_0=\max\{\|u\|_\lambda,\|v\|_\lambda\}$.
Therefore the first term on the right-hand side of \eqref{eq:Phi_diff_start} is bounded by
\[
|u-v|_{\lambda,1}\sum_{n=1}^\infty |c_n|\,n\,R_0^{n-1}
= \tilde{\Phi}'(R_0)\,|u-v|_{\lambda,1}.
\]

For the second term, fix $n\ge 2$. By the triangle inequality and \eqref{eq:algebra_dlambda},
for each $0\le m\le n-1$,
\[
|u^m v^{n-1-m}|_{\lambda,1}
\le |u^m|_{\lambda,1}\,\|v^{n-1-m}\|_\lambda+\|u^m\|_\lambda\,|v^{n-1-m}|_{\lambda,1}.
\]
Using $\|u^m\|_\lambda\le R_0^m$, $\|v^{n-1-m}\|_\lambda\le R_0^{n-1-m}$ and
$|u^m|_{\lambda,1}\le m\,R_0^{m-1}|u|_{\lambda,1}$ (and similarly for $v$), we obtain
\[
|u^m v^{n-1-m}|_{\lambda,1}
\le \big(m\,|u|_{\lambda,1}+(n-1-m)\,|v|_{\lambda,1}\big)\,R_0^{n-2}.
\]
Summing over $m=0,\dots,n-1$ yields
\[
\Big|\sum_{m=0}^{n-1}u^m v^{n-1-m}\Big|_{\lambda,1}
\le n(n-1)\,R_0^{n-2}\big(|u|_{\lambda,1}+|v|_{\lambda,1}\big).
\]
Hence the second term on the right-hand side of \eqref{eq:Phi_diff_start} is bounded by
\[
\|u-v\|_\lambda\sum_{n=2}^\infty |c_n|\,n(n-1)\,R_0^{n-2}\big(|u|_{\lambda,1}+|v|_{\lambda,1}\big)
=\tilde{\Phi}''(R_0)\,(\,|u|_{\lambda,1}+|v|_{\lambda,1}\,)\,\|u-v\|_\lambda.
\]
Combining the two bounds yields the claim.
\end{proof}

\medskip
\subsection{Auxiliary analytic functions}
Consider a weight function $\omega$ given by
\begin{equation} \label{weight}
\omega(v)= \dfrac{1}{\pi(1+v^2)},
\end{equation}
and note that $\int \omega(v) \, \rd v = 1$.
Let $\alpha = \alpha(v)$ be given by
\begin{equation} \label{alpha}
\alpha(v) = \frac{\partial_v \omega(v)}{\omega(v)} = - \frac{2v}{1+v^2}.
\end{equation}
Both functions $\alpha, \omega$ are analytic with radius of convergence $\Lambda = 1$, since by direct computation one has
\begin{equation*}
\|\partial_v^\ell \omega\|_\infty\le \frac{\ell!}{\pi},\qquad
\|\partial_v^\ell\alpha\|_\infty\le 2\ell!\qquad(\ell\ge 0).
\end{equation*}
Hence, by \eqref{eq:analytic_finite_H} we have 
\begin{equation} \label{alpha0}
\alpha_0 \coloneqq \| \alpha \|_{H,\lambda_0} < \infty,
\end{equation}
and
\begin{equation} \label{beta0_beta1}
\beta_0 \coloneqq \| \omega \|_{H,\lambda_0} < \infty, \qquad \beta_1 \coloneqq | \omega |_{H,\lambda_0} < \infty,  
\end{equation}
for $0< \lambda_0 < 1$.

\section{Local well-posedness: perturbations of the trivial solution} \label{section_main}

In this section we establish a local-in-time analytic well-posedness result for \eqref{vlasov0}
by a Banach fixed-point argument in a time-dependent analytic framework.
For technical reasons we work with the weighted unknown
\[
g=\frac{f}{\omega},
\]
where $\omega$ is defined in \eqref{weight} and $\alpha=\partial_v\omega/\omega$ is given by \eqref{alpha}.
If $f$ solves \eqref{vlasov0}, then $g$ formally solves
\begin{equation}\label{eq:g_equation}
\begin{dcases}
\partial_t g + v \cdot \partial_x g - \partial_x\Phi(\rho)\cdot(\partial_v g + \alpha g) = 0, \\
\rho = \displaystyle\int_{\R} \omega\, g \,\rd v, \\
g|_{t=0} = g_0 := f_0/\omega.
\end{dcases}
\end{equation}
Conversely, if $g$ solves \eqref{eq:g_equation}, then $f=\omega g$ solves \eqref{vlasov0}.
\par
Consider positive parameters $\lambda_0$, $K$ and $T$ satisfying:
\begin{equation} \label{parameters}
0 < T < 1, \qquad  T < \lambda_0 < 1, \qquad  0 < K < \frac{\lambda_0}{T} - 1.
\end{equation}
Let $\lambda: [0,T] \to \R$ be given by
\begin{equation} \label{lambda}
\lambda(t) = \lambda_0 - (K+1)t.
\end{equation}
Additionally, for $M > 0$ consider the following set $\mathrm{X}_{\lambda_0, K, T}^M$:
\begin{equation} \label{spaceX}
\mathrm{X}_{\lambda_0, K, T}^M
=\left\lbrace g \in C\big([0,T]; C^\infty(\mathbb{R}^2)\big) : \,
\sup_{t\in[0,T]} \|g(t) \|_{H, \lambda(t)} + \int_0^T |g(t)|_{H, \lambda(t)} \, \rd t  \leq M \right\rbrace,
\end{equation}
endowed with the metric induced by
\begin{equation} \label{normZ}
\| g\|_{\Z} =  \sup_{t\in[0,T]} \|g(t) \|_{ \lambda(t)} + \int_0^T |g(t)|_{\lambda(t),1} \, \rd t.
\end{equation}
\begin{remark}\label{rem:rho_well_defined}
Our analytic framework is based on sup-type analytic norms and does not require $f(t)\in L^1(\R_x\times\R_v)$.
However, for any $g\in \mathrm{X}_{\lambda_0,K,T}^M$ we have, for all $t\in[0,T]$,
\[
\|g(t)\|_\infty \le \|g(t)\|_{\lambda(t)} \le \|g(t)\|_{H,\lambda(t)} \le M,
\]
hence $g$ is bounded on $[0,T]\times\R^2$. Therefore the density
\[
\rho(t,x)=\int_{\R}\omega(v)\,g(t,x,v)\,\rd v
\]
is well-defined for all $(t,x)\in[0,T]\times\R$ since $\omega\in L^1(\R_v)$, and consequently
\[
f=\omega g \in L^\infty\big([0,T]\times\R_x;L^1(\R_v)\big).
\]
No decay as $|x|\to\infty$ is imposed, so the associated mass/energy may be infinite.
\end{remark}
\par
The local-in-time analytic well-posedness result is stated as follows.

\begin{theorem}\label{mainthm}
Let $\Phi$ be real-analytic on $(-R,R)$ in the sense of Definition~\ref{def:entire}.
Fix parameters $\lambda_0,K,T$ satisfying \eqref{parameters}, and let $M\in(0,R)$.
Assume that $M$ satisfies
\begin{align} \label{M_condition}
2(1+\alpha_0)M\Big(\tilde{\Phi}'(M)+2MT\,\tilde{\Phi}''(M)\Big)\,e^{\alpha_0\tilde{\Phi}'(M)MT} < 1
\leq K-\lambda_0-5\,\tilde{\Phi}'(M)M,
\end{align}
and that the initial datum $f_0 \in C^\infty(\mathbb{R}^2)$ is such that
\begin{equation} \label{f_0size}
\| \pi(1+v^2)f_0 \|_{H, \lambda_0}
\leq \frac{M}{2} e^{-(5+\alpha_0)\tilde{\Phi}'(M)M},
\end{equation}
where $\alpha_0$ is as in \eqref{alpha0}.
Then there exists a unique solution $g \in \mathrm{X}_{\lambda_0, K, T}^M$ to \eqref{eq:g_equation}.
In particular, $f:=\omega g$ is a solution of \eqref{vlasov0} on $[0,T]$ with $f|_{t=0}=f_0$, where $\omega$ is as in~\eqref{weight},
and it is unique among solutions $f$ such that $f/\omega \in \mathrm{X}_{\lambda_0,K,T}^M$.
\end{theorem}

\begin{remark}\label{rem:interpret_assumptions}
The conditions in Theorem~\ref{mainthm} ensure two requirements: (i) the analyticity radius
$\lambda(t)=\lambda_0-(K+1)t$ remains positive on $[0,T]$ (so the time-dependent analytic norms are well-defined), and
(ii) the fixed point map $\Psi$ constructed below maps the ball $\mathrm{X}_{\lambda_0,K,T}^M$ into itself and is a contraction.
In particular, the hypothesis \eqref{f_0size} is a smallness condition on the initial datum $g_0 = f_0/\omega$, so the theorem
is a local-in-time small-data result (a perturbation around the trivial steady state~$g \equiv 0$, hence $f\equiv 0$).
\end{remark}

\begin{remark}\label{rem:f_in_X_ball}
Let $g\in \mathrm{X}_{\lambda_0,K,T}^M$ be the unique solution to \eqref{eq:g_equation} given by Theorem~\ref{mainthm} and let $f=\omega g$. Since $\lambda(t)\le \lambda_0$ for $t\in[0,T]$, we have
$\|\omega\|_{H,\lambda(t)}\le \beta_0$ and $|\omega|_{H,\lambda(t)}\le \beta_1$, where $\beta_0,\beta_1$ are as in~\eqref{beta0_beta1}. Using~\eqref{eq:algebraH}--\eqref{eq:algebra_dlambdaH}, we obtain for all $t\in[0,T]$,
\[
\|f(t)\|_{H,\lambda(t)}\le \beta_0\,\|g(t)\|_{H,\lambda(t)},
\qquad
|f(t)|_{H,\lambda(t)}
\le 2\beta_0\,|g(t)|_{H,\lambda(t)}+2\beta_1\,\|g(t)\|_{H,\lambda(t)}.
\]
Consequently,
\[
\sup_{t\in[0,T]}\|f(t)\|_{H,\lambda(t)}
+\int_0^T |f(t)|_{H,\lambda(t)}\,\rd t
\le 2(\beta_0+\beta_1 T)M,
\]
so $f\in \mathrm{X}_{\lambda_0,K,T}^{\widetilde M}$ with $\widetilde M:=2(\beta_0+\beta_1 T)M$.
\end{remark}

\par
We now introduce the fixed point formulation of \eqref{eq:g_equation}. Given a coefficient field $\sigma$, the map $\Psi$ below is defined by solving a linear transport equation with smooth coefficients; for background on such quasilinear/transport systems we refer to \cite{rodest1983systems}.
Define $\Psi:\mathrm{X}_{\lambda_0,K,T}^M\to \mathrm{X}_{\lambda_0,K,T}^M$ by:
\begin{equation}\label{vlasov}
\begin{split}
\text{Given }&h\in \mathrm{X}_{\lambda_0, K, T}^M,\ \text{let }\Psi(h)=g \text{ be the solution of}\\
&
\begin{dcases}
\partial_t g + v \cdot \partial_x g - \partial_x\Phi(\sigma)\cdot(\partial_v g + \alpha g) = 0, \\
\sigma = \int_{\R} \omega \, h \, \rd v, \\
g|_{t=0} = g_0 = f_0 / \omega.
\end{dcases}
\end{split}
\end{equation}
For $h\in \mathrm{X}_{\lambda_0,K,T}^M$ we have $\|\sigma(t)\|_\infty\le \|\sigma(t)\|_{H,\lambda(t)}\le M$,
hence $\sigma(t,x)\in[-M,M]\subseteq(-R,R)$ and $\Phi(\sigma)$ is well-defined.
A fixed point $g$ of $\Psi$ yields a solution of \eqref{eq:g_equation}, leading to a solution $f= \omega g$ of~\eqref{vlasov0}.

\par 
For completeness, we verify that the space $\mathrm{X}_{\lambda_0, K, T}^M$ used in the Banach fixed-point argument is complete. To do so, we show that it is a closed subset of a Banach space $\mathrm{Z}_{\lambda_0, K, T}$ defined as
\begin{equation} \label{spaceZ}
\mathrm{Z}_{\lambda_0, K, T} = \left\lbrace g \in C\big([0,T]; C^\infty(\mathbb{R}^2)\big)  :\,  \|g \|_\Z = \sup_{t\in[0,T]} \|g(t) \|_{\lambda(t)} + \int_0^T |g(t)|_{\lambda(t),1} \, \rd t  < \infty \right\rbrace.
\end{equation} 
In Appendix \ref{section_complete_metric_space} we provide a proof that $\mathrm{Z}_{\lambda_0, K, T}$ is complete. To check that $\mathrm{X}_{\lambda_0, K, T}^M$ is a closed subset of $\mathrm{Z}_{\lambda_0, K, T}$, we take a sequence $(g_m) \subseteq \mathrm{X}_{\lambda_0, K, T}^M$ such that $\| g_m-g \|_\Z \to 0$ as $m \to \infty$ for some $g \in \mathrm{Z}_{\lambda_0, K, T}$. By the definition of $\mathrm{X}_{\lambda_0, K, T}^M $, for every $m \in \mathbb{N}$ we have
\begin{equation} \label{aux1}
\sup_{t\in[0,T]} \|g_m(t) \|_{H, \lambda(t)} + \int_0^T |g_m(t)|_{H, \lambda(t)} \, \rd t  \leq M.
\end{equation}
Moreover, from the assumed convergence we have for all $t \in [0,T]$ and every $k,\ell \geq 0$ that
\[\|\partial_x^k \partial_v^\ell\big(g_m(t) - g(t) \big) \|_\infty \to 0 \quad  \text{as} \ m \to \infty \]
which implies  
\[\|\partial_x^k \partial_v^\ell g_m(t) \|_\infty \to \|\partial_x^k \partial_v^\ell g(t) \|_\infty \quad \text{as} \ m \to \infty.\]
By Fatou's lemma we deduce for each $t \in [0,T]$,
\begin{align*}
\|g(t) \|_{H,\lambda(t)} & = \sum_{n \geq 0} \frac{1}{(n!)^2} \sum_{k + \ell \geq n} \frac{(k+\ell)!}{(k+\ell-n)!} \frac{\lambda(t)^{k+\ell-n}}{k! \ell!} \|\partial_x^k \partial_v^\ell g(t) \|_{\infty} \\ 
& = \sum_{n \geq 0} \frac{1}{(n!)^2} \sum_{k + \ell \geq n} \frac{(k+\ell)!}{(k+\ell-n)!} \frac{\lambda(t)^{k+\ell-n}}{k! \ell!} \lim_{m \to \infty} \|\partial_x^k \partial_v^\ell g_m(t) \|_{\infty} \\
& \leq \liminf_{m \to \infty} \|g_m(t) \|_{H,\lambda(t)},
\end{align*}
and similarly, 
\begin{align}
\int_0^T |g(t)|_{H, \lambda(t)} \, \rd t \leq  \liminf_{m \to \infty} \int_0^T |g_m(t)|_{H, \lambda(t)} \, \rd t,
\end{align}
from which by (\ref{aux1}) we see
\begin{align}
\sup_{t\in[0,T]} \|g(t) \|_{H, \lambda(t)} + \int_0^T |g(t)|_{H, \lambda(t)} \, \rd t \leq M,
\end{align}
and so $g \in \mathrm{X}_{\lambda_0, K, T}^M$, as desired.

\subsection{Preliminary estimates}\label{section_prelimestimates}

\begin{lemma}\label{estimate1prop}
Let $g=g(t,x,v)$ be an analytic solution of \eqref{vlasov}, and set
\[
F(t,x)=- \partial_x\Phi(\sigma(t,x)).
\]
Then for each $k,\ell\in\N_0$,
\begin{equation}\label{estimate1}
\frac{\rd}{\rd t}\,\|\partial_x^k\partial_v^\ell g\|_\infty
\le \ell\,\|\partial_x^{k+1}\partial_v^{\ell-1}g\|_\infty + I_1+I_2,
\end{equation}
where
\begin{equation}\label{I1I2}
\begin{split}
I_1
&= \sum_{i=0}^{k-1} {k\choose i}\,
\|\partial_x^{k-i}F\|_\infty\,\|\partial_x^i\partial_v^{\ell+1}g\|_\infty,\\
I_2
&= \sum_{i=0}^{k} {k\choose i}\,
\|\partial_x^{k-i}F\|_\infty
\sum_{j=0}^{\ell} {\ell\choose j}\,
\|\partial_v^{\ell-j}\alpha\|_\infty\,\|\partial_x^i\partial_v^{j}g\|_\infty.
\end{split}
\end{equation}
\end{lemma}

\begin{proof}
We apply the operator $\partial_x^k \partial_v^\ell$ to (\ref{vlasov}). The first term simply becomes $\partial_t (\partial_x^k \partial_v^\ell g)$, while the second is computed as follows:
\begin{align*}
\partial_x^k \partial_v^\ell (v \cdot \partial_x g) & = \partial_v^\ell (v\cdot \partial_x^{k+1}g) \\ 
& = \sum_{j=0}^\ell {\ell \choose j} \partial_v^j v \cdot \partial_v^{\ell-j} \partial_x^{k+1}g \\
& = v \cdot \partial_x (\partial_x^k \partial_v^\ell g) + \ell \partial_x^{k+1} \partial_v^{\ell-1} g.
\end{align*}
For the force term, note that $F=F(t,x)$ does not depend on $v$, hence
\begin{align*}
\partial_x^k\partial_v^\ell\big(F\cdot\partial_v g\big)
&=\partial_x^k \big(F\cdot\partial_v^{\ell+1}g\big)
= F\cdot\partial_v(\partial_x^k\partial_v^\ell g)
+ \sum_{i=0}^{k-1}{k\choose i}\,\partial_x^{k-i}F\cdot\partial_x^i\partial_v^{\ell+1}g,
\\
\partial_x^k\partial_v^\ell\big(F\cdot\alpha g\big)
&=\sum_{i=0}^k {k\choose i}\,\partial_x^{k-i}F\cdot
\partial_x^i\partial_v^\ell(\alpha g)
\\
&=\sum_{i=0}^k {k\choose i}\,\partial_x^{k-i}F\,
\sum_{j=0}^\ell {\ell\choose j}\,\partial_v^{\ell-j}\alpha\cdot\partial_x^i\partial_v^j g.
\end{align*}
Collecting terms, the quantity $u=\partial_x^k\partial_v^\ell g$ satisfies the following kinetic equation:
\begin{align*}
\partial_t u  + v & \cdot \partial_x u + F\cdot\partial_v u \\
 = \ & - \ell \partial_x^{k+1} \partial_v^{\ell-1} g - \sum_{i=0}^{k-1}{k\choose i}\,\partial_x^{k-i}F\cdot\partial_x^i\partial_v^{\ell+1}g \\ & -\sum_{i=0}^k {k\choose i}\,\partial_x^{k-i}F\,
\sum_{j=0}^\ell {\ell\choose j}\,\partial_v^{\ell-j}\alpha\cdot\partial_x^i\partial_v^j g.
\end{align*}
Applying the maximum principle
\eqref{maximumprinciple} yields \eqref{estimate1}.
\end{proof}

\begin{lemma}\label{estimate2prop}
Let $g=g(t,x,v)$ be an analytic solution of \eqref{vlasov} and 
\[
F(t,x)=-\partial_x\Phi(\sigma(t,x)).
\]
Then for each $\lambda>0$ and $n\in\N_0$,
\begin{equation}\label{estimate2}
\frac{\rd}{\rd t}|g|_{\lambda,n}
\le \lambda\,|g|_{\lambda,n+1}+n\,|g|_{\lambda,n}
+\frac{\rd^n}{\rd\lambda^n}\Big( \|F\|_\lambda\,|g|_{\lambda,1}
+\|F\|_\lambda\,\|\alpha\|_\lambda\,\|g\|_\lambda \Big).
\end{equation}
\end{lemma}

\begin{proof}
We first apply $\dfrac{\rd^n}{\rd \lambda^n}\!\left( \dfrac{\lambda^{k+\ell}}{k!\, \ell!}\right)$
to the inequality \eqref{estimate1} and then sum over $k,\ell \in \mathbb{N}_0$ to obtain
\begin{equation}\label{eq:est2_start}
\frac{\rd}{\rd t} |g|_{\lambda,n} \leq
\dfrac{\rd^n}{\rd \lambda^n} \sum_{\substack{k \geq 0 \\ \ell \geq 1}} \frac{\ell \lambda^{k+\ell}}{k!\, \ell!}
\| \partial_x^{k+1} \partial_v^{\ell-1} g \|_\infty + J_1 + J_2,
\end{equation}
where
\[
J_1 = \sum_{k,\ell \geq 0} \dfrac{\rd^n}{\rd \lambda^n}\!\left(\frac{\lambda^{k+\ell}}{k!\, \ell!}\right) I_1,
\qquad
J_2 = \sum_{k,\ell \geq 0} \dfrac{\rd^n}{\rd \lambda^n}\!\left(\frac{\lambda^{k+\ell}}{k!\, \ell!}\right) I_2,
\]
with $I_1$ and $I_2$ given in \eqref{I1I2}.
\par

We now estimate the three terms on the right-hand side of \eqref{eq:est2_start}. The first term is treated as follows:
\begin{align*}
\dfrac{\rd^n}{\rd \lambda^n} \sum_{\substack{k \geq 0 \\ \ell \geq 1}}
\frac{\ell \lambda^{k+\ell}}{k!\, \ell!} \| \partial_x^{k+1} \partial_v^{\ell-1} g \|_\infty
&= \dfrac{\rd^n}{\rd \lambda^n} \sum_{k,\ell \geq 0}
\frac{\lambda^{k+\ell+1}}{k!\,\ell!} \| \partial_x^{k+1}\partial_v^\ell g \|_\infty \\
&= \dfrac{\rd^n}{\rd \lambda^n}\big( \lambda \|\partial_x g \|_\lambda \big) \\
& = \sum_{m=0}^n {n \choose m} \frac{\rd^m}{\rd \lambda^m} \lambda \cdot \frac{\rd^{n-m}}{\rd \lambda^{n-m}} \| \partial_x g\|_\lambda \\
&= \lambda |\partial_x g|_{\lambda,n} + n |\partial_x g|_{\lambda, n-1} \\
& \leq \lambda |g|_{\lambda, n+1} + n |g|_{\lambda,n},
\end{align*}
where in the last step we used \eqref{eq:dx_shift} (and for the $n=0$ case we interpret
$|\partial_x g|_{\lambda,-1}=0$). 
\par 
Regarding the term $J_1$, we have:
\begin{align*}
\sum_{k,\ell \geq 0}\frac{\lambda^{k+\ell}}{k!\,\ell!}\,I_1
&=
\sum_{k\ge 1}\sum_{\ell\ge 0}\frac{\lambda^{k+\ell}}{k!\,\ell!}
\sum_{i=0}^{k-1}{k\choose i}\,
\|\partial_x^{k-i}F\|_\infty\,\|\partial_x^i\partial_v^{\ell+1}g\|_\infty\\
&=
\sum_{m\ge 1}\frac{\lambda^{m}}{m!}\|\partial_x^{m}F\|_\infty\,
\sum_{i,\ell\ge 0}\frac{\lambda^{i+\ell}}{i!\,\ell!}\|\partial_x^i\partial_v^{\ell+1}g\|_\infty\\
&\le \|F\|_\lambda\,\|\partial_v g\|_\lambda \\
&\le \|F\|_\lambda\, |g|_{\lambda,1},
\end{align*}
where in the last step we used \eqref{eq:dx_shift}. Applying $\frac{\rd^n}{\rd\lambda^n}$ yields
\[
J_1 \le \frac{\rd^n}{\rd\lambda^n}\big(\|F\|_\lambda\,|g|_{\lambda,1}\big).
\]
Similarly,
\begin{align*}
\sum_{k,\ell\ge 0}\frac{\lambda^{k+\ell}}{k!\,\ell!}\,I_2
&=
\sum_{k,\ell\ge 0}\frac{\lambda^{k+\ell}}{k!\,\ell!}
\sum_{i=0}^{k} {k\choose i}\,
\|\partial_x^{k-i}F\|_\infty
\sum_{j=0}^{\ell} {\ell\choose j}\,
\|\partial_v^{\ell-j}\alpha\|_\infty\,\|\partial_x^i\partial_v^j g\|_\infty\\
&=
\sum_{m\ge 0}\frac{\lambda^{m}}{m!}\|\partial_x^{m}F\|_\infty\,
\sum_{p\ge 0}\frac{\lambda^{p}}{p!}\|\partial_v^{p}\alpha\|_\infty\,
\sum_{i,j\ge 0}\frac{\lambda^{i+j}}{i!\,j!}\|\partial_x^i\partial_v^j g\|_\infty\\
&=\|F\|_\lambda\,\|\alpha\|_\lambda\,\|g\|_\lambda.
\end{align*}
Applying $\frac{\rd^n}{\rd\lambda^n}$ yields
\[
J_2 = \frac{\rd^n}{\rd\lambda^n}\big(\|F\|_\lambda\,\|\alpha\|_\lambda\,\|g\|_\lambda\big).
\]
Combining the previous estimates with \eqref{eq:est2_start} proves \eqref{estimate2}.
\end{proof}

\medskip
\subsection{A time-dependent estimate}
~\par
Recall that $\lambda(t)=\lambda_0-(K+1)t$ for $t\in[0,T]$, with $\lambda_0>(1+K)T$.
\begin{lemma}\label{lem:timeestimate}
If $g=g(t,x,v)$ is an analytic solution of \eqref{vlasov} and
\[
F(t,x)=-\partial_x\Phi(\sigma(t,x)),
\]
then
\begin{equation}\label{timeestimate}
\begin{split}
\frac{\rd}{\rd t}\,\|g(t)\|_{H,\lambda(t)} \leq &\ (\lambda_0-K)\,|g(t)|_{H,\lambda(t)} \\
&\ + \sum_{n \geq 0} \frac{1}{(n!)^2}\,
\dfrac{\rd^n}{\rd \lambda^n}\Big( \|F(t)\|_{\lambda}\,|g(t)|_{\lambda,1}\Big)\Big|_{\lambda=\lambda(t)} \\
&\ + \sum_{n \geq 0} \frac{1}{(n!)^2}\,
\dfrac{\rd^n}{\rd \lambda^n}\Big( \|F(t)\|_{\lambda}\,\|\alpha\|_{\lambda}\,\|g(t)\|_{\lambda}\Big)\Big|_{\lambda=\lambda(t)} .
\end{split}
\end{equation}
\end{lemma}

\begin{proof}
By definition \eqref{normsH} and the chain rule,
\begin{align*}
\frac{\rd}{\rd t}\|g(t)\|_{H,\lambda(t)}
&=\frac{\rd}{\rd t}\sum_{n\ge 0}\frac{1}{(n!)^2}\,|g(t)|_{\lambda(t),n}\\
&=\sum_{n\ge 0}\frac{\lambda'(t)}{(n!)^2}\,|g(t)|_{\lambda(t),n+1}
+\sum_{n\ge 0}\frac{1}{(n!)^2}\,\frac{\rd}{\rd t}\Big(|g(t)|_{\lambda,n}\Big)\Big|_{\lambda=\lambda(t)}.
\end{align*}
Applying Lemma~\ref{estimate2prop} with $\lambda=\lambda(t)$ yields
\begin{align*}
\frac{\rd}{\rd t}\|g(t)\|_{H,\lambda(t)} \leq\,
&\sum_{n\ge 0}\frac{\lambda'(t)+\lambda(t)}{(n!)^2}\,|g(t)|_{\lambda(t),n+1}
+\sum_{n\ge 0}\frac{n}{(n!)^2}\,|g(t)|_{\lambda(t),n}\\
&\quad+\sum_{n\ge 0}\frac{1}{(n!)^2}\,
\dfrac{\rd^n}{\rd \lambda^n}\Big( \|F(t)\|_{\lambda}\,|g(t)|_{\lambda,1}\Big)\Big|_{\lambda=\lambda(t)}\\
&\quad+\sum_{n\ge 0}\frac{1}{(n!)^2}\,
\dfrac{\rd^n}{\rd \lambda^n}\Big( \|F(t)\|_{\lambda}\,\|\alpha\|_{\lambda}\,\|g(t)\|_{\lambda}\Big)\Big|_{\lambda=\lambda(t)}.
\end{align*}
We now estimate the first two sums. Since $\lambda'(t)=-(K+1)$, we have \[\lambda'(t)+\lambda(t)+1=\lambda(t)-K.\] Therefore,
\begin{align*}
&\sum_{n\ge 0}\frac{\lambda'(t)+\lambda(t)}{(n!)^2}\,|g(t)|_{\lambda(t),n+1}
+\sum_{n\ge 0}\frac{n}{(n!)^2}\,|g(t)|_{\lambda(t),n} \\
&=(\lambda'(t)+\lambda(t))\sum_{n\ge 1}\frac{n^2}{(n!)^2}\,|g(t)|_{\lambda(t),n}
+\sum_{n\ge 1}\frac{n}{(n!)^2}\,|g(t)|_{\lambda(t),n} \\
&\le (\lambda'(t)+\lambda(t))\,|g(t)|_{H,\lambda(t)} + |g(t)|_{H,\lambda(t)} \\
&= (\lambda(t)-K)\,|g(t)|_{H,\lambda(t)} \\
&\le (\lambda_0-K)\,|g(t)|_{H,\lambda(t)}.
\end{align*}
Combining this with the previous inequality gives \eqref{timeestimate}.
\end{proof}

\medskip
\subsection{Supplementary lemmas}

The first lemma provides an estimate for the second term on the right-hand side of
\eqref{timeestimate}, namely the contribution involving $\|F\|_\lambda\,|g|_{\lambda,1}$. 

\begin{lemma}\label{lem:force_term_Phi}
Let $\Phi$ be real-analytic on $(-R,R)$ with expansion $\Phi(z)=\sum_{n\ge0}c_n z^n$ for $|z|<R$. Let $g=g(x,v)$ and $\sigma=\sigma(x)$ be analytic, and set
\[
F(x)=-\partial_x\Phi(\sigma(x)).
\]
Then, for every $\lambda>0$ such that $\|\sigma\|_{H,\lambda}<R$,
\begin{equation}\label{eq:force_term_Phi}
\sum_{n \ge 0}\frac{1}{(n!)^2}\,
\frac{\rd^n}{\rd\lambda^n}\Big(\|F\|_{\lambda}\,|g|_{\lambda,1}\Big)
\le 5\,\tilde{\Phi}'(\|\sigma\|_{H,\lambda})\,
\Big(\|\sigma\|_{H,\lambda}\,|g|_{H,\lambda}
+|\sigma|_{H,\lambda}\,\|g\|_{H,\lambda}\Big).
\end{equation}
\end{lemma}

\begin{proof}
We use the identity
\[
F=-(\Phi'(\sigma))\,\partial_x\sigma.
\]
By \eqref{eq:algebra_conv} applied to the product $\Phi'(\sigma)\cdot \partial_x\sigma$ and
\eqref{eq:dx_shift}, for every $n\in\N_0$, we have
\begin{equation}\label{eq:F_lambda_coeff}
\begin{split}
|F|_{\lambda,n}=|\Phi'(\sigma)\,\partial_x\sigma|_{\lambda,n}
& \le \sum_{m=0}^n\binom{n}{m}\,|\Phi'(\sigma)|_{\lambda,m}\,|\partial_x\sigma|_{\lambda,n-m} \\
& \le \sum_{m=0}^n\binom{n}{m}\,|\Phi'(\sigma)|_{\lambda,m}\,|\sigma|_{\lambda,n-m+1}.
\end{split}
\end{equation}
We now estimate the left-hand side of \eqref{eq:force_term_Phi}.
By the Leibniz rule,
\[
\frac{\rd^n}{\rd\lambda^n}\Big(\|F\|_{\lambda}\,|g|_{\lambda,1}\Big)
=\sum_{p=0}^n \binom{n}{p}\,|F|_{\lambda,p}\,|g|_{\lambda,n-p+1}.
\]
Using \eqref{eq:F_lambda_coeff} and changing the summation order yields
\begin{align*}
\frac{\rd^n}{\rd\lambda^n}\Big(\|F\|_{\lambda}\,|g|_{\lambda,1}\Big)
&\le \sum_{p=0}^n \binom{n}{p}\sum_{m=0}^p\binom{p}{m}\,|\Phi'(\sigma)|_{\lambda,m}\,|\sigma|_{\lambda,p-m+1}\,|g|_{\lambda,n-p+1}\\
&=\sum_{\substack{m,r,s\ge 0\\ m+r+s=n}}
\binom{n}{m}\binom{n-m}{r}\,|\Phi'(\sigma)|_{\lambda,m}\,|\sigma|_{\lambda,r+1}\,|g|_{\lambda,s+1}.
\end{align*}
Multiplying by $\frac{1}{(n!)^2}$ and summing over $n\ge 0$, we obtain
\begin{equation}\label{eq:force_start}
\begin{split}
\sum_{n\ge 0} &\frac{1}{(n!)^2}\,
\frac{\rd^n}{\rd\lambda^n}\Big(\|F\|_{\lambda}\,|g|_{\lambda,1}\Big) \\
& \le
\sum_{m,r,s\ge 0}\frac{1}{((m+r+s)!)^2}\binom{m+r+s}{m}\binom{r+s}{r}\,
|\Phi'(\sigma)|_{\lambda,m}\,|\sigma|_{\lambda,r+1}\,|g|_{\lambda,s+1}.
\end{split}
\end{equation}
We use the elementary bound
\begin{equation}\label{eq:triple_factorial_bound}
\begin{split}
\frac{1}{((m+r+s)!)^2}\binom{m+r+s}{m}\binom{r+s}{r} &=\frac{1}{(m+r+s)!}\frac{1}{m!\,r!\,s!} \\
&\le \frac{1}{(m!)^2}\frac{1}{((r+s)!)^2}\binom{r+s}{r}.
\end{split}
\end{equation}
Inserting \eqref{eq:triple_factorial_bound} into \eqref{eq:force_start} yields
\begin{equation}\label{eq:pull_out_AH_final}
\begin{split}
\sum_{n\ge 0}&\frac{1}{(n!)^2}\,
\frac{\rd^n}{\rd\lambda^n}\Big(\|F\|_{\lambda}\,|g|_{\lambda,1}\Big)
\\ 
&\le
\Big(\sum_{m\ge 0}\frac{|\Phi'(\sigma)|_{\lambda,m}}{(m!)^2}\Big)
\sum_{r,s\ge 0}\frac{1}{((r+s)!)^2}\binom{r+s}{r}\,|\sigma|_{\lambda,r+1}\,|g|_{\lambda,s+1}.
\end{split}
\end{equation}
The first factor is $\|\Phi'(\sigma)\|_{H,\lambda}$. Since $\Phi'(z)=\sum_{n\ge 1}n c_n z^{n-1}$ and $\|\!\cdot\!\|_{H,\lambda}$ is an algebra norm
(Lemma~\ref{lemma_algebraH}),
\[
\|\Phi'(\sigma)\|_{H,\lambda}
\le \sum_{n\ge 1}n|c_n|\,\|\sigma^{\,n-1}\|_{H,\lambda}
\le \sum_{n\ge 1}n|c_n|\,\|\sigma\|_{H,\lambda}^{\,n-1}
=\tilde{\Phi}'(\|\sigma\|_{H,\lambda}).
\]
Thus,
\begin{equation}\label{eq:Phi_prime_pullout}
\begin{split}
\sum_{n\ge 0}&\frac{1}{(n!)^2}\,
\frac{\rd^n}{\rd\lambda^n}\Big(\|F\|_{\lambda}\,|g|_{\lambda,1}\Big)
\le
\tilde{\Phi}'(\|\sigma\|_{H,\lambda})\, S.
\end{split}
\end{equation}
where
\[
S=\sum_{r,s\ge 0}\frac{1}{((r+s)!)^2}\binom{r+s}{r}\,|\sigma|_{\lambda,r+1}\,|g|_{\lambda,s+1}.
\]
\par 
The remainder of the proof is devoted to estimate $S$, which we rewrite as follows: 
\[
S=\sum_{n,m\ge 0}\frac{|\sigma|_{\lambda,n+1}}{((n+1)!)^2}\,
\frac{|g|_{\lambda,m+1}}{((m+1)!)^2}\,A(n,m),
\]
where
\[
A(n,m)=\binom{n+m}{n}\,\frac{((n+1)!)^2((m+1)!)^2}{((n+m)!)^2}.
\]
We start by the contributions of $n=0,1$. We compute
\[
A(0,m)=(m+1)^2,\qquad
A(1,m)=\frac{4(m+1)^2}{m+2}\le 2(m+1)^2.
\]
Hence
\begin{align*}
\sum_{m\ge 0}\frac{|\sigma|_{\lambda,1}}{(1!)^2}\frac{|g|_{\lambda,m+1}}{((m+1)!)^2}A(0,m)
&=|\sigma|_{\lambda,1}\sum_{m\ge 0}\frac{(m+1)^2}{((m+1)!)^2}|g|_{\lambda,m+1} \\
& =|\sigma|_{\lambda,1}\,|g|_{H,\lambda} \\
& \le \|\sigma\|_{H,\lambda}\,|g|_{H,\lambda},
\end{align*}
and
\begin{align*}
\sum_{m\ge 0}\frac{|\sigma|_{\lambda,2}}{(2!)^2}\frac{|g|_{\lambda,m+1}}{((m+1)!)^2}A(1,m)
&\le \frac{|\sigma|_{\lambda,2}}{4}\,2\sum_{m\ge 0}\frac{(m+1)^2}{((m+1)!)^2}|g|_{\lambda,m+1} \\
& =2\frac{|\sigma|_{\lambda,2}}{4}\,|g|_{H,\lambda}\\
&\le 2\,\|\sigma\|_{H,\lambda}\,|g|_{H,\lambda},
\end{align*}
Therefore, the total $n=0,1$ contribution is bounded by
\[
3\,\|\sigma\|_{H,\lambda}\,|g|_{H,\lambda},
\]
and, by symmetry, the total $m=0,1$ contribution is bounded by
\[
3\,|\sigma|_{H,\lambda}\,\|g\|_{H,\lambda}.
\]
Therefore,
\begin{equation} \label{eq:S_aux1}
S \leq 3 \, (\|\sigma\|_{H,\lambda}\,|g|_{H,\lambda} + |\sigma|_{H,\lambda}\,\|g\|_{H,\lambda}) + \sum_{n,m\ge 2}\frac{|\sigma|_{\lambda,n+1}}{((n+1)!)^2}\,
\frac{|g|_{\lambda,m+1}}{((m+1)!)^2}\,A(n,m). 
\end{equation}
Now, we claim that
\begin{equation}\label{eq:A_tail_bound}
A(n,m)\le 18\qquad \text{for all} \ m,n\ge 2.
\end{equation}
Indeed, if $n=2$ then
\[
A(2,m)=18\,\frac{m+1}{m+2}<18,
\]
and similarly $A(n,2)<18$. Moreover, for fixed $m\ge 3$ we have
\[
\frac{A(n+1,m)}{A(n,m)}
=\frac{(n+2)^2}{(n+1)(n+m+1)}
\le 1
\qquad (n\ge 2,\ m\ge 3),
\]
since $(n+1)(n+m+1)-(n+2)^2=n(m-2)+(m-3)\ge 0$.
Thus for $m\ge 3$ the map $n\mapsto A(n,m)$ is nonincreasing on $\{2,3,\dots\}$, hence
$A(n,m)\le A(2,m)<18$, and by symmetry we obtain \eqref{eq:A_tail_bound}.

Using \eqref{eq:A_tail_bound},
\[
\sum_{m,n\ge 2}\frac{|\sigma|_{\lambda,n+1}}{((n+1)!)^2}\,
\frac{|g|_{\lambda,m+1}}{((m+1)!)^2}\,A(n,m)
\le 18\,B_\sigma\,B_g,
\]
where
\[
B_\sigma=\sum_{n\ge 2}\frac{|\sigma|_{\lambda,n+1}}{((n+1)!)^2},
\qquad
B_g=\sum_{m\ge 2}\frac{|g|_{\lambda,m+1}}{((m+1)!)^2}.
\]
Since $n+1\ge 3$ on $B_\sigma$, we have $(n+1)^2\ge 9$ and therefore
\[
B_\sigma\le \frac{1}{9}\sum_{n\ge 2}\frac{(n+1)^2}{((n+1)!)^2}|\sigma|_{\lambda,n+1}
\le \frac{1}{9}\,|\sigma|_{H,\lambda},
\]
and similarly
\[
B_g\le \sum_{m\ge 2}\frac{|g|_{\lambda,m+1}}{((m+1)!)^2}\le \|g\|_{H,\lambda}.
\]
Hence
\[
18\,B_\sigma B_g \le 2\,|\sigma|_{H,\lambda}\,\|g\|_{H,\lambda}.
\]
Interchanging the roles (using $B_g\le \frac{1}{9}|g|_{H,\lambda}$ and $B_\sigma\le \|\sigma\|_{H,\lambda}$) also gives
\[
18\,B_\sigma B_g \le 2\,\|\sigma\|_{H,\lambda}\,|g|_{H,\lambda}.
\]
Therefore the contribution $n,m\geq2$ is bounded by
\[
2\Big(\|\sigma\|_{H,\lambda}\,|g|_{H,\lambda}+|\sigma|_{H,\lambda}\,\|g\|_{H,\lambda}\Big),
\]
which combined with \eqref{eq:S_aux1} yields
\[
S\le 5\Big(\|\sigma\|_{H,\lambda}\,|g|_{H,\lambda}+|\sigma|_{H,\lambda}\,\|g\|_{H,\lambda}\Big),
\]
and finishes the proof.
\end{proof}

Compared with the analogous estimate in \cite[Lemma~2.11]{bossy2013local}, our argument yields a smaller numerical constant (5 instead of 16).\par
The next lemma provides an estimate for the third term on the right-hand side of \eqref{timeestimate}.

\begin{lemma}\label{lemmaII_new}
Let $g$, $F$ and $\alpha$ be analytic functions. Then
\begin{equation}\label{lemmaII_new_eq}
\sum_{n \geq 0} \frac{1}{(n!)^2}\,
\dfrac{\rd^n}{\rd \lambda^n}\Big( \|F\|_{\lambda}\,\|\alpha\|_{\lambda}\,\|g\|_{\lambda} \Big)
\leq \|F\|_{H,\lambda}\,\|\alpha\|_{H,\lambda}\,\|g\|_{H,\lambda}.
\end{equation}
\end{lemma}

\begin{proof}
By Leibniz rule,
\[
\dfrac{\rd^n}{\rd \lambda^n}\Big( \|F\|_{\lambda}\,\|\alpha\|_{\lambda}\,\|g\|_{\lambda} \Big)
=\sum_{\substack{m_1+m_2+m_3=n\\ m_i\ge 0}}
{n\choose m_1,m_2,m_3}\,|F|_{\lambda,m_1}\,|\alpha|_{\lambda,m_2}\,|g|_{\lambda,m_3}.
\]
Therefore,
\begin{align*}
\sum_{n \geq 0} & \frac{1}{(n!)^2}\,
\dfrac{\rd^n}{\rd \lambda^n}\Big( \|F\|_{\lambda}\,\|\alpha\|_{\lambda}\,\|g\|_{\lambda} \Big) \\
&=
\sum_{n \ge 0}\frac{1}{(n!)^2}
\sum_{\substack{m_1+m_2+m_3=n\\ m_i\ge 0}}
\frac{n!}{m_1!\,m_2!\,m_3!}\,
|F|_{\lambda,m_1}\,|\alpha|_{\lambda,m_2}\,|g|_{\lambda,m_3}\\
&=
\sum_{m_1,m_2,m_3\ge 0}
\frac{1}{(m_1+m_2+m_3)!}\,
\frac{1}{m_1!\,m_2!\,m_3!}\,
|F|_{\lambda,m_1}\,|\alpha|_{\lambda,m_2}\,|g|_{\lambda,m_3}.
\end{align*}
Since $(m_1+m_2+m_3)!\ge m_1!\,m_2!\,m_3!$, we have
\[
\frac{1}{(m_1+m_2+m_3)!}\,\frac{1}{m_1!\,m_2!\,m_3!}
\le \frac{1}{(m_1!)^2}\,\frac{1}{(m_2!)^2}\,\frac{1}{(m_3!)^2}.
\]
Hence
\begin{align*}
\sum_{n \geq 0} &\frac{1}{(n!)^2}\,
\dfrac{\rd^n}{\rd \lambda^n}\Big( \|F\|_{\lambda}\,\|\alpha\|_{\lambda}\,\|g\|_{\lambda} \Big) \\
&\le
\sum_{m_1,m_2,m_3\ge 0}
\frac{|F|_{\lambda,m_1}}{(m_1!)^2}\,
\frac{|\alpha|_{\lambda,m_2}}{(m_2!)^2}\,
\frac{|g|_{\lambda,m_3}}{(m_3!)^2}\\
&=
\sum_{m_1\ge 0}\frac{|F|_{\lambda,m_1}}{(m_1!)^2}
\sum_{m_2\ge 0}\frac{|\alpha|_{\lambda,m_2}}{(m_2!)^2}
\sum_{m_3\ge 0}\frac{|g|_{\lambda,m_3}}{(m_3!)^2}\\
&=\|F\|_{H,\lambda}\,\|\alpha\|_{H,\lambda}\,\|g\|_{H,\lambda},
\end{align*}
which proves \eqref{lemmaII_new_eq}.
\end{proof}

\medskip
\subsection{Contraction mapping}
~\par
Given $h \in \X_{\lambda_0, K, T}^M$, we prove that under some restrictions on the positive parameters $\lambda_0, K,T, M$ and on the initial data, the function $g = \Psi(h)$ belongs to $\X_{\lambda_0, K, T}^M$  and $\Psi$ is a contraction for the metric induced by the norm $\| \idot \|_\Z$. This establishes the existence and uniqueness result stated in Theorem \ref{mainthm}.

\begin{lemma}\label{lem:F_bounds}
Let $\Phi$ be real-analytic on $(-R,R)$ with expansion $\Phi(z)=\sum_{n\ge0}c_n z^n$ for $|z|<R$.
Let $\sigma=\sigma(x)$ be analytic and set $F=-\partial_x\Phi(\sigma)$.
Then, for every $\lambda>0$ such that $\|\sigma\|_{\lambda}<R$,
\begin{equation}\label{eq:F_lambda_bound}
\|F\|_{\lambda}\le \tilde{\Phi}'(\|\sigma\|_{\lambda})\,|\sigma|_{\lambda,1}.
\end{equation}
Moreover, for every $\lambda>0$ such that $\|\sigma\|_{H,\lambda}<R$,
\begin{equation}\label{eq:F_H_bound}
\|F\|_{H,\lambda}\le \tilde{\Phi}'(\|\sigma\|_{H,\lambda})\,|\sigma|_{H,\lambda}.
\end{equation}
\end{lemma}

\begin{proof}
We use the identity
\[
F=-\partial_x\Phi(\sigma)=-(\Phi'(\sigma))\,\partial_x\sigma.
\]
Then, by \eqref{eq:dx_shift} and \eqref{eq:algebra},
\[
\|F\|_{\lambda}\le \|\Phi'(\sigma)\|_{\lambda}\,\|\partial_x\sigma\|_{\lambda}
\le\|\Phi'(\sigma)\|_{\lambda}\,|\sigma|_{\lambda,1}.
\]
Applying~\eqref{composition_est} to $\Phi'$ gives
$\|\Phi'(\sigma)\|_{\lambda}\le \tilde{\Phi}'(\|\sigma\|_{\lambda})$, hence~\eqref{eq:F_lambda_bound}.
\par
Regarding inequality~\eqref{eq:F_H_bound}, we first note that $\|\idot\|_{H,\lambda}$ is an algebra norm: for analytic $u,v$,
\begin{equation}\label{eq:H_algebra}
\|uv\|_{H,\lambda}\le \|u\|_{H,\lambda}\,\|v\|_{H,\lambda}.
\end{equation}
Indeed, for each $n\ge 0$,
\[
|uv|_{\lambda,n}\le\sum_{m=0}^n{n\choose m}|u|_{\lambda,m}|v|_{\lambda,n-m},
\]
and summing with weights $\frac{1}{(n!)^2}$ leads to \eqref{eq:H_algebra}: using \[\frac{1}{(n!)^2}\binom{n}{m}
=\frac{1}{n!}\frac{1}{m!(n-m)!}\le \frac{1}{(m!)^2}\frac{1}{((n-m)!)^2},\] we obtain
\[
\sum_{n\ge 0}\frac{1}{(n!)^2}|uv|_{\lambda,n}
\le \sum_{m\ge 0}\frac{|u|_{\lambda,m}}{(m!)^2}
     \sum_{r\ge 0}\frac{|v|_{\lambda,r}}{(r!)^2}
=\|u\|_{H,\lambda}\|v\|_{H,\lambda}.
\]
Hence,
\[
\|F\|_{H,\lambda}
\le \|\Phi'(\sigma)\|_{H,\lambda}\,\|\partial_x\sigma\|_{H,\lambda}.
\]
By \eqref{eq:dx_shift} applied inside the definition of $\|\idot\|_{H,\lambda}$,
\[
\|\partial_x\sigma\|_{H,\lambda}
=\sum_{n\ge 0}\frac{1}{(n!)^2}\,|\partial_x\sigma|_{\lambda,n}
\le \sum_{n\ge 0}\frac{1}{(n!)^2}\,|\sigma|_{\lambda,n+1}
=\sum_{m\ge 1}\frac{m^2}{(m!)^2}\,|\sigma|_{\lambda,m}
=|\sigma|_{H,\lambda}.
\]
Next, since $\Phi'(z)=\sum_{n\ge 1}n c_n z^{n-1}$ and $\|\idot\|_{H,\lambda}$ is an algebra norm,
\[
\|\Phi'(\sigma)\|_{H,\lambda}
\le \sum_{n\ge 1}n|c_n|\,\|\sigma^{\,n-1}\|_{H,\lambda}
\le \sum_{n\ge 1}n|c_n|\,\|\sigma\|_{H,\lambda}^{n-1}
=\tilde{\Phi}'(\|\sigma\|_{H,\lambda}).
\]
Combining the previous inequalities yields \eqref{eq:F_H_bound}.
\end{proof}

We are now in conditions to establish the first main block towards the proof of Theorem~\ref{mainthm}. The next result yields that $\Psi$ is well defined on $\mathrm{X}_{\lambda_0,K,T}^M$.

\begin{proposition}\label{psiwelldefinedprop}
Let $\Phi$ be real-analytic on $(-R,R)$ for some $R>M>0$. Let $\lambda_0>(K+1)T$ and suppose that
\begin{equation}\label{eq:psi_invariance_cond}
K-\lambda_0-5\,\tilde{\Phi}'(M)M\geq1
\quad \text{and} \quad
\|g_0\|_{H,\lambda_0}\le \frac{M}{2} e^{-(5+\alpha_0)\tilde{\Phi}'(M)M},
\end{equation}
where $\alpha_0=\|\alpha\|_{H,\lambda_0}$.
Then
\begin{equation}\label{psiwelldefined}
\Psi(\mathrm{X}_{\lambda_0,K,T}^M)\subseteq \mathrm{X}_{\lambda_0,K,T}^M.
\end{equation}
\end{proposition}

\begin{proof}
Let $h\in \mathrm{X}_{\lambda_0,K,T}^M$ and  $\sigma=\int_{\R}\omega\,h\,\rd v$.
By the definition of $\mathrm{X}_{\lambda_0,K,T}^M$ and since $\int_{\R}\omega\,\rd v=1$, we have
\[
\sup_{t\in[0,T]}\|\sigma(t)\|_{H,\lambda(t)}+\int_0^T |\sigma(t)|_{H,\lambda(t)}\,\rd t\le M.
\]
Let $g=\Psi(h)$ solve \eqref{vlasov} and let $F(t,x)=-\partial_x\Phi(\sigma(t,x))$.

Applying Lemma~\ref{lem:timeestimate} and Lemmas~\ref{lem:force_term_Phi}--\ref{lemmaII_new}, we obtain for each $t\in[0,T]$,
\begin{align*}
\frac{\rd}{\rd t}\|g(t)\|_{H,\lambda(t)}
\le\ &(\lambda(t)-K)\,|g(t)|_{H,\lambda(t)} \\
&+5\,\tilde{\Phi}'(\|\sigma(t)\|_{H,\lambda(t)})\,
\Big(\|\sigma(t)\|_{H,\lambda(t)}\,|g(t)|_{H,\lambda(t)}
+|\sigma(t)|_{H,\lambda(t)}\,\|g(t)\|_{H,\lambda(t)}\Big)\\
&+\|F(t)\|_{H,\lambda(t)}\,\|\alpha\|_{H,\lambda(t)}\,\|g(t)\|_{H,\lambda(t)}.
\end{align*}
Since $\lambda(t)\le\lambda_0$, we have $\|\alpha\|_{H,\lambda(t)}\le \|\alpha\|_{H,\lambda_0}=\alpha_0$.
Moreover, $\|\sigma(t)\|_{H,\lambda(t)}\le M$, and $\tilde{\Phi}'$ is nondecreasing,
hence $\tilde{\Phi}'(\|\sigma(t)\|_{H,\lambda(t)})\le \tilde{\Phi}'(M)$. Using also Lemma~\ref{lem:F_bounds},
\[
\|F(t)\|_{H,\lambda(t)}\le \tilde{\Phi}'(\|\sigma(t)\|_{H,\lambda(t)})\,|\sigma(t)|_{H,\lambda(t)}
\le \tilde{\Phi}'(M)\,|\sigma(t)|_{H,\lambda(t)}.
\]
Therefore,
\begin{equation}\label{eq:psi_timeineq_new}
\frac{\rd}{\rd t}\|g(t)\|_{H,\lambda(t)}
\le -A\,|g(t)|_{H,\lambda(t)} + B(t)\,\|g(t)\|_{H,\lambda(t)},
\end{equation}
where
\[
A=K-\lambda_0-5\,\tilde{\Phi}'(M)\,M,
\qquad
B(t)=(5+\alpha_0)\,\tilde{\Phi}'(M)\,|\sigma(t)|_{H,\lambda(t)}.
\]
Note that $A\geq1$ by \eqref{eq:psi_invariance_cond}, and that
\[
\int_0^T B(t)\,\rd t \le (5+\alpha_0)\,\tilde{\Phi}'(M)\int_0^T |\sigma(t)|_{H,\lambda(t)}\,\rd t
\le (5+\alpha_0)\,\tilde{\Phi}'(M)\,M.
\]
Let $E(t)$ be the integrating factor
\[
E(t)=\exp\Big(-\int_0^t B(\tau)\,\rd\tau\Big)\in(0,1].
\]
Multiplying \eqref{eq:psi_timeineq_new} by $E(t)$ gives
\[
\frac{\rd}{\rd t}\big(E(t)\,\|g(t)\|_{H,\lambda(t)}\big)\le -A\,E(t)\,|g(t)|_{H,\lambda(t)},
\]
which upon integration from $0$ to $t$ leads to
\[
E(t)\,\|g(t)\|_{H,\lambda(t)} + A\int_0^t E(\tau)\,|g(\tau)|_{H,\lambda(\tau)}\,\rd\tau
\le \|g_0\|_{H,\lambda_0} \quad (t \in [0,T]).
\]
Since $E(t)\ge E(T)=\exp\!\big(-\int_0^T B\big)$, we obtain for all $t\in[0,T]$,
\[
\|g(t)\|_{H,\lambda(t)}
\le \frac{1}{E(T)}\,\|g_0\|_{H,\lambda_0}
= \exp\!\Big(\int_0^T B\Big)\,\|g_0\|_{H,\lambda_0}
\le e^{(5+\alpha_0)\tilde{\Phi}'(M)M}\,\|g_0\|_{H,\lambda_0}.
\]
Moreover, since $A\geq 1$, we may drop $A$ to obtain
\[
\int_0^T E(t)\,|g(t)|_{H,\lambda(t)}\,\rd t \le \|g_0\|_{H,\lambda_0}.
\]
Using again $E(t)\ge E(T)$, it follows that
\[
\int_0^T |g(\tau)|_{H,\lambda(\tau)}\,\rd\tau
\le \frac{1}{E(T)}\,\|g_0\|_{H,\lambda_0}
\le e^{(5+\alpha_0)\tilde{\Phi}'(M)M}\,\|g_0\|_{H,\lambda_0}.
\]
Combining the two bounds yields
\[
\sup_{t\in[0,T]}\|g(t)\|_{H,\lambda(t)}
+\int_0^T |g(\tau)|_{H,\lambda(\tau)}\,\rd\tau
\le 2\,e^{(5+\alpha_0)\tilde{\Phi}'(M)M}\,\|g_0\|_{H,\lambda_0}.
\]
Finally, the smallness condition on $\|g_0\|_{H,\lambda_0}$ in \eqref{eq:psi_invariance_cond}
implies that the right-hand side is controlled by $M$, hence $g\in \mathrm{X}_{\lambda_0,K,T}^M$ and \eqref{psiwelldefined} follows.
\end{proof}

At last, we prove that under the hypothesis of Theorem~\ref{mainthm}, $\Psi$ is a contraction. 

\begin{proposition}\label{prop:psi_contraction}
In addition to the hypotheses of Proposition~\ref{psiwelldefinedprop}, assume that
\begin{equation}\label{eq:psi_contraction_cond}
\kappa = 2(1+\alpha_0)\,M\,\Big(\tilde{\Phi}'(M)+2MT\,\tilde{\Phi}''(M)\Big)\,e^{\alpha_0\tilde{\Phi}'(M)MT} < 1.
\end{equation}
Then $\Psi:\mathrm{X}_{\lambda_0,K,T}^M\to \mathrm{X}_{\lambda_0,K,T}^M$ is a contraction for the metric
induced by the norm $\|\idot \|_{\Z}$.
\end{proposition}

\begin{proof}
Let $h,\bar h\in \mathrm{X}_{\lambda_0,K,T}^M$, $\sigma=\int_{\R}\omega\,h\,\rd v$, $\bar\sigma=\int_{\R}\omega\,\bar h\,\rd v,$ and
\[
g=\Psi(h), \qquad \bar g=\Psi(\bar h), \qquad F=-\partial_x\Phi(\sigma),\qquad \bar F=-\partial_x\Phi(\bar\sigma).
\]
Subtracting the two linear problems \eqref{vlasov} gives
\begin{equation}\label{eq:d_eq}
\partial_t (g - \bar g) + v\cdot \partial_x (g - \bar g) + F\cdot\big(\partial_v (g - \bar g) + \alpha (g - \bar g)\big)
= S,
\end{equation}
where
\[
S=(\bar F-F)\cdot(\partial_v \bar g+\alpha\,\bar g).
\]

\smallskip
\noindent\textbf{First step.}
Arguing as in Lemma~\ref{estimate2prop} with $n=0$, we obtain
for every fixed $\lambda>0$:
\begin{align*}
\frac{\rd}{\rd t}\|g(t) - \bar g(t)\|_{\lambda}
\le & \ \lambda\,|g(t) - \bar g(t)|_{\lambda,1}+\|F(t)\|_{\lambda}\,|g(t) - \bar g(t)|_{\lambda,1} \\
 & +  \|F(t)\|_{\lambda}\,\|\alpha\|_{\lambda}\,\|g(t) - \bar g(t)\|_{\lambda}+\|S(t)\|_{\lambda}.
\end{align*}
Using $\lambda=\lambda(t)$ and the chain rule,
\[
\frac{\rd}{\rd t}\|g(t) - \bar g(t)\|_{\lambda(t)}
=\lambda'(t)\,|g(t) - \bar g(t)|_{\lambda(t),1}+\frac{\rd}{\rd t}\|g(t) - \bar g(t)\|_{\lambda}\Big|_{\lambda=\lambda(t)}.
\]
Since $\lambda'(t)=-(K+1)$ and $\lambda(t)\le \lambda_0$, we have
\[
\lambda'(t)+\lambda(t)=\lambda(t)-(K+1)\le \lambda_0-K-1.
\]
Hence
\begin{equation}\label{eq:d_basic_ineq}
\begin{split}
\frac{\rd}{\rd t}\|g(t) - \bar g(t)\|_{\lambda(t)}
\le & \ (\lambda_0-K-1+\|F(t)\|_{\lambda(t)}) \,|g(t) - \bar g(t)|_{\lambda(t),1} \\
 & +  \|F(t)\|_{\lambda(t)}\,\|\alpha\|_{\lambda(t)}\,\|g(t) - \bar g(t)\|_{\lambda(t)}+\|S(t)\|_{\lambda(t)}.
 \end{split}
\end{equation}

\smallskip
\noindent\textbf{Second step.}
Since $h\in \mathrm{X}_{\lambda_0,K,T}^M$, we have
\[
\|\sigma(t)\|_{\lambda(t)}\le \|\sigma(t)\|_{H,\lambda(t)}\le M,
\qquad
|\sigma(t)|_{\lambda(t),1}\le \|\sigma(t)\|_{H,\lambda(t)}\le M,
\]
and hence, by \eqref{eq:F_lambda_bound},
\begin{equation}\label{eq:F_unif_bound}
\|F(t)\|_{\lambda(t)}
\le \tilde{\Phi}'(\|\sigma(t)\|_{\lambda(t)})\,|\sigma(t)|_{\lambda(t),1}
\le \tilde{\Phi}'(M)M.
\end{equation}
Also, since $\lambda(t)\le \lambda_0$ and $\|\idot\|_{H,\lambda}$ dominates $\|\idot\|_{\lambda}$,
\begin{equation}\label{eq:alpha_unif_bound}
\|\alpha\|_{\lambda(t)}\le \|\alpha\|_{H,\lambda(t)}\le \|\alpha\|_{H,\lambda_0}=\alpha_0.
\end{equation}
By hypothesis, $K-\lambda_0-5\tilde{\Phi}'(M)M\ge 1$, and so
$K-\lambda_0-\tilde{\Phi}'(M)M\ge 1$, hence
\[
\lambda_0-K-1+\|F(t)\|_{\lambda(t)}\le \lambda_0-K-1+\tilde{\Phi}'(M)M\le -1.
\]
Thus \eqref{eq:d_basic_ineq}--\eqref{eq:alpha_unif_bound} yield
\begin{equation}\label{eq:d_gronwall_ready}
\frac{\rd}{\rd t}\|g(t) - \bar g(t)\|_{\lambda(t)} + |g(t) - \bar g(t)|_{\lambda(t),1}
\le \alpha_0\tilde{\Phi}'(M)M \,\|g(t) - \bar g(t)\|_{\lambda(t)}+\|S(t)\|_{\lambda(t)}.
\end{equation}

\smallskip
\noindent\textbf{Third step.}
By the algebra property of $\|\idot\|_{\lambda}$ and \eqref{eq:dx_shift}, we deduce
\[
\|S(t)\|_{\lambda(t)}
\le \|\bar F(t)-F(t)\|_{\lambda(t)}\,
\big(|\bar g(t)|_{\lambda(t),1}+\|\alpha\|_{\lambda(t)}\,\|\bar g(t)\|_{\lambda(t)}\big).
\]
Since $\bar g\in \mathrm{X}_{\lambda_0,K,T}^M$ we have
\[
|\bar g(t)|_{\lambda(t),1}\le \|\bar g(t)\|_{H,\lambda(t)}\le M,
\qquad
\|\bar g(t)\|_{\lambda(t)}\le \|\bar g(t)\|_{H,\lambda(t)}\le M.
\]
Together with \eqref{eq:alpha_unif_bound}, this implies
\begin{equation}\label{eq:S_bound_by_Fdiff}
\|S(t)\|_{\lambda(t)}\le (1+\alpha_0)\,M\,\|\bar F(t)-F(t)\|_{\lambda(t)}.
\end{equation}
Next, since $\bar F-F=-\partial_x(\Phi(\bar\sigma)-\Phi(\sigma))$, \eqref{eq:dx_shift} gives
\[
\|\bar F(t)-F(t)\|_{\lambda(t)}
\le |\Phi(\bar\sigma(t))-\Phi(\sigma(t))|_{\lambda(t),1}.
\]
Set $R_0(t)=\max\{\|\sigma(t)\|_{\lambda(t)},\|\bar\sigma(t)\|_{\lambda(t)}\}\le M$.
Applying Lemma~\ref{lem:Phi_diff_dlambda} yields
\[
|\Phi(\bar\sigma(t))-\Phi(\sigma(t))|_{\lambda(t),1}
\le \tilde{\Phi}'(M)\,|\bar\sigma(t)-\sigma(t)|_{\lambda(t),1}
+2M\,\tilde{\Phi}''(M)\,\|\bar\sigma(t)-\sigma(t)\|_{\lambda(t)}.
\]
Since $\sigma-\bar\sigma$ is an $\omega$-average,
\[
\|\sigma(t)-\bar\sigma(t)\|_{\lambda(t)}\le \|h(t)-\bar h(t)\|_{\lambda(t)},
\qquad
|\sigma(t)-\bar\sigma(t)|_{\lambda(t),1}\le |h(t)-\bar h(t)|_{\lambda(t),1}.
\]
Combining these estimates with \eqref{eq:S_bound_by_Fdiff} gives, for all $t\in[0,T]$,
\[
\|S(t)\|_{\lambda(t)}
\le (1+\alpha_0)\,M\,\Big(\tilde{\Phi}'(M)\,|h(t)-\bar h(t)|_{\lambda(t),1}
+2M\,\tilde{\Phi}''(M)\,\|h(t)-\bar h(t)\|_{\lambda(t)}\Big).
\]
Integrating in time yields
\begin{equation}\label{eq:intS_bound}
\int_0^T \|S(t)\|_{\lambda(t)}\,\rd t
\le (1+\alpha_0)\,M\,\Big(\tilde{\Phi}'(M)+2MT\,\tilde{\Phi}''(M)\Big)\,\|h-\bar h\|_{\Z}.
\end{equation}

\smallskip
\noindent\textbf{Fourth step.}
Set $a=\alpha_0\tilde{\Phi}'(M)M$ and
\[
y(t)=\|g(t) - \bar g(t)\|_{\lambda(t)},\quad z(t):=|g(t) - \bar g(t)|_{\lambda(t),1},\quad s(t):=\|S(t)\|_{\lambda(t)}.
\]
Then \eqref{eq:d_gronwall_ready} reads
\begin{equation}\label{eq:yz_ineq}
y'(t)+z(t)\le a\,y(t)+s(t),\qquad y(0)=0.
\end{equation}
Multiply \eqref{eq:yz_ineq} by the integrating factor $e^{-at}$ to obtain
\[
\frac{\rd}{\rd t}\big(e^{-at}y(t)\big)+e^{-at}z(t)\le e^{-at}s(t).
\]
Integrating from $0$ to $t$ yields
\begin{equation}\label{eq:int_factor_est}
e^{-at}y(t)+\int_0^t e^{-a\tau}z(\tau)\,\rd\tau
\le \int_0^t e^{-a\tau}s(\tau)\,\rd\tau.
\end{equation}
From \eqref{eq:int_factor_est} we immediately deduce
\[
y(t)\le e^{at}\int_0^t e^{-a\tau}s(\tau)\,\rd\tau
\le e^{aT}\int_0^T s(\tau)\,\rd\tau,
\]
hence
\begin{equation}\label{eq:sup_y}
\sup_{t\in[0,T]}\|g(t) - \bar g(t)\|_{\lambda(t)}
\le e^{aT}\int_0^T \|S(\tau)\|_{\lambda(\tau)}\,\rd\tau.
\end{equation}
Moreover, taking $t=T$ in \eqref{eq:int_factor_est} and using $e^{-a\tau}\ge e^{-aT}$ on $[0,T]$,
\[
e^{-aT}\int_0^T z(\tau)\,\rd\tau
\le \int_0^T e^{-a\tau}z(\tau)\,\rd\tau
\le \int_0^T e^{-a\tau}s(\tau)\,\rd\tau
\le \int_0^T s(\tau)\,\rd\tau,
\]
so
\begin{equation}\label{eq:int_z}
\int_0^T |g(t) - \bar g(t)|_{\lambda(t),1}\,\rd t
\le e^{aT}\int_0^T \|S(t)\|_{\lambda(t)}\,\rd t.
\end{equation}
Combining \eqref{eq:sup_y}--\eqref{eq:int_z} gives
\begin{align*}
\|g - \bar g\|_{\Z}
&=\sup_{t\in[0,T]}\|g(t) - \bar g(t)\|_{\lambda(t)}+\int_0^T |g(t) - \bar g(t)|_{\lambda(t),1}\,\rd t \\
& \le 2e^{aT}\int_0^T \|S(t)\|_{\lambda(t)}\,\rd t.
\end{align*}
Using \eqref{eq:intS_bound}, we conclude that
\[
\|\Psi(h)-\Psi(\bar h)\|_{\Z}=\|g - \bar g\|_{\Z}
\le \kappa\|h-\bar h\|_{\Z},
\]
where
\[
\kappa=2(1+\alpha_0)\,M\,\Big(\tilde{\Phi}'(M)+2MT\,\tilde{\Phi}''(M)\Big)\,e^{\alpha_0\,\tilde{\Phi}'(M)\,M\,T}.
\]
This proves the contraction property provided $\kappa<1$.
\end{proof}

\subsection{Nonemptiness of the parameter assumptions}

\begin{lemma}\label{lem:nonempty_parameters}
Assume that $\Phi$ is real-analytic on $(-R,R)$ in the sense of Definition~\ref{def:entire}, and let $\alpha_0=\|\alpha\|_{H,\lambda_0}$ for some fixed
$\lambda_0\in(0,1)$, where $\alpha$ is as in~\eqref{alpha}. Then there exist $M\in(0,R)$, $K>0$ and $T\in(0,1)$ such that
\begin{equation}\label{eq:nonempty_goal}
0<T<1,\qquad T<\lambda_0<1,\qquad 0<K<\frac{\lambda_0}{T}-1,
\end{equation}
and
\begin{equation}\label{eq:nonempty_Mcond}
2(1+\alpha_0)M\Big(\tilde{\Phi}'(M)+2MT\,\tilde{\Phi}''(M)\Big)\,e^{\alpha_0\tilde{\Phi}'(M)MT} < 1
\le K-\lambda_0-5\,\tilde{\Phi}'(M)M.
\end{equation}
In particular, the hypotheses of Theorem~\ref{mainthm} are not empty.
\end{lemma}

\begin{proof}
Fix $\lambda_0\in(0,1)$. Since $\Phi(z)=\sum_{n\ge0}c_n z^n$ has radius of convergence $R$, the series defining
$\tilde\Phi',\tilde\Phi''$ converge for every $r\in[0,R)$ (see Remark~\ref{rem:majorant}).
In particular, on the compact interval $[0,R/2]$ the functions $\tilde\Phi',\tilde\Phi''$ are finite and bounded; set
\[
C_1:=\sup_{0\le r\le R/2}\tilde\Phi'(r)<\infty,
\qquad
C_2:=\sup_{0\le r\le R/2}\tilde\Phi''(r)<\infty.
\]
\par
For $M\in(0,R/2]$ and any $T\in(0,1)$ we have the bounds
\[
\tilde\Phi'(M)\le C_1,\qquad \tilde\Phi''(M)\le C_2,\qquad e^{\alpha_0\tilde\Phi'(M)MT}\le e^{\alpha_0 C_1 M}.
\]
Hence
\begin{align*}
2(1+\alpha_0)M\Big(\tilde{\Phi}'(M)+2MT\,\tilde{\Phi}''(M)\Big)\,e^{\alpha_0\tilde{\Phi}'(M)MT}
&\le
2(1+\alpha_0)M\big(C_1+2M C_2\big)e^{\alpha_0 C_1 M}.
\end{align*}
The right-hand side tends to $0$ as $M\downarrow 0$. Therefore we can choose $M\in(0,R/2]$ so small that
\begin{equation*}
2(1+\alpha_0)M\big(C_1+2M C_2\big)e^{\alpha_0 C_1 M} < 1.
\end{equation*}
For this choice of $M$, the strict inequality in \eqref{eq:nonempty_Mcond} holds for every $T\in(0,1)$.
\par 
Next, define
\[
K:=\lambda_0+2+5\,\tilde\Phi'(M)M.
\]
Then
\[
K-\lambda_0-5\,\tilde\Phi'(M)M=2\ge 1,
\]
so the second inequality in \eqref{eq:nonempty_Mcond} holds.
\par 
Finally, we pick
\[
T=\min\Big\{\frac12, \ \frac{\lambda_0}{2(K+1)}\Big\}.
\]
Then $0<T<1$ and $T<\lambda_0$. Moreover, $(K+1)T\le \lambda_0/2<\lambda_0$ which implies
\[
K<\frac{\lambda_0}{T}-1,
\]
which is exactly the last condition in \eqref{eq:nonempty_goal}. Thus \eqref{eq:nonempty_goal} holds.
\end{proof}

\section{Perturbations around nontrivial equilibrium solutions}\label{sec:perturb_nontrivial}

In this section we focus on time-independent (stationary) solutions of~\eqref{vlasov0}. In particular, we explain how the local-in-time analytic theory from Section~\ref{section_main} extends to small perturbations around spatially homogeneous stationary profiles.

A function $f\in L^1_{\rm loc}(\R_x \times \R_v)$ is a \emph{weak stationary state} of~\eqref{vlasov0} provided
$f(x,\cdot)\in L^1(\R_v)$ for a.e.\ $x\in\R$,
$\Phi(\rho)\in W^{1,\infty}_{\rm loc}(\R)$, and the equation
\begin{equation}\label{eq:stat_vlasov}
v\cdot\partial_x f-\partial_x\Phi(\rho)\cdot\partial_v f=0
\end{equation}
holds in the sense of distributions.
If in addition $v \mapsto v f(x,v)\in L^1(\R_v)$ and $f(x,v)\to 0$ as $|v|\to\infty$ for every $x$, then integrating~\eqref{eq:stat_vlasov} in $v$ yields that the flux
\[
J(x)\coloneqq \int v f(x,v)\,\rd v
\]
is constant. When $J\equiv 0$, the stationary state is referred to as an \emph{equilibrium}. \par

\begin{remark}\label{rem:homogeneous_steady_states}
Let $f^\ast \in L^1(\R_v)$ be independent of $x$ and set $\rho^\ast=\int f^\ast(v)\,\rd v$.
Then $\rho^\ast$ is constant and therefore $\partial_x\Phi(\rho^\ast)=0$. Hence $f(x,v)= f^\ast(v)$ is
a weak stationary state of~\eqref{vlasov0}.
\end{remark}

Next, we introduce the (local) Hamiltonian energy function
\begin{equation}\label{eq:station_H}
H(x,v)=\frac{v^2}{2}+\Phi(\rho(x)).
\end{equation}
Assuming sufficient smoothness, \emph{e.g.}\ $\Phi'(\rho) \, \partial_x \rho \in C(\R)$, the stationary transport field in~\eqref{eq:stat_vlasov} has characteristics
\begin{equation}\label{eq:stat_charact}
\dot x(s)=v(s),\qquad \dot v(s)=-\partial_x\Phi(\rho(x(s))).
\end{equation}
Along any such curve $(x(s),v(s))$, the quantity $H(x(s),v(s))$ is constant, since $\frac{\rd}{\rd s}H(x(s),v(s))=0$.
\par 
The next result provides a characterization of weak stationary states which are $C^1$ on upper and lower half (phase) spaces.

For $h\in\R$, set
\begin{equation} \label{eq:set_A(h)}
A(h):=\{x\in\R:\Phi(\rho(x))<h\}.
\end{equation}
Since $A(h)$ is an open subset of $\R$, it is a countable union of disjoint open intervals.
We denote its connected components by $\{E_k(h)\}_{k\in I(h)}$, where $I(h)\subseteq\N$.
For $x\in A(h)$, let $K(x,h)\in I(h)$ be the unique index such that $x\in E_{K(x,h)}(h)$.

\begin{proposition}\label{prop:Fk_Gk}
Let $f=f(x,v)$ be $C^1$ on $\{v < 0\} \cup \{v > 0\}$ and with $\rho = \int f\,\rd v\in C^1(\R)$. 

Then $f$ is a weak stationary state of~\eqref{vlasov0} if and only if the following hold:
\begin{enumerate}[(i)]
\item\label{it:rep_components}
There exist families of functions $\{F_k\}_{k\in\N}$ and $\{G_k\}_{k\in\N}$ such that, for all $x\in\R$,
\begin{equation} \label{eq:Fk_Gk}
f(x,v)=
\begin{dcases}
F_{K(x,H(x,v))}\!\bigl(H(x,v)\bigr), & v>0,\\
G_{K(x,H(x,v))}\!\bigl(H(x,v)\bigr), & v<0,
\end{dcases}
\end{equation}
where $H=H(x,v)$ is as in~\eqref{eq:station_H}.
\smallskip
\item\label{it:rho_components}
For every $x\in\R$,
\begin{equation}
\begin{split}
\rho(x)= & \int_0^\infty
F_{K\left(x,\frac{v^2}{2}+\Phi(\rho(x))\right)}\left(\frac{v^2}{2}+\Phi(\rho(x))\right)\rd v \\
& +
\int_0^\infty
G_{K\left(x,\frac{v^2}{2}+\Phi(\rho(x))\right)}\left(\frac{v^2}{2}+\Phi(\rho(x))\right)\rd v.
\end{split}
\end{equation}
\smallskip
\item\label{it:jump_condition_components}
For a.e.\ $x\in\R$,
\begin{equation}
\lim_{v \to 0^+} \partial_x\Phi(\rho(x)) \, F_{K(x,H(x,v))}(H(x,v)) = \lim_{v \to 0^-} \partial_x\Phi(\rho(x)) \, G_{K(x,H(x,v))}(H(x,v)).
\end{equation}
\end{enumerate}
\end{proposition}

\begin{proof}
Set $\Omega_+:=\R\times(0,\infty)$ and consider
\[
T_+:\Omega_+\to\R^2,\qquad T_+(x,v):=(x,h),\qquad h:=H(x,v)=\frac{v^2}{2}+\Phi(\rho(x)).
\]
Since $\partial_v H(x,v)=v>0$ on $\Omega_+$, the map $T_+$ is a $C^1$ diffeomorphism from $\Omega_+$ onto
\[
\Sigma_+:=\bigl\{(x,h)\in\R^2:\ h>\Phi(\rho(x))\bigr\},
\]
with inverse
\[
T_+^{-1}(x,h)=\Bigl(x,\sqrt{2\bigl(h-\Phi(\rho(x))\bigr)}\Bigr).
\]
Define $\tilde f_+:\Sigma_+\to\R$ by $\tilde f_+(x,h):=f\bigl(T_+^{-1}(x,h)\bigr)$. On $\Omega_+$, a direct chain-rule computation yields
\[
v\,\partial_x f-\partial_x\Phi(\rho)\,\partial_v f
=
v\,\partial_x\tilde f_+(x,h),
\qquad (x,h)=T_+(x,v).
\]
Hence, if $f$ is stationary on $\Omega_+$, then $\partial_x\tilde f_+=0$ on $\Sigma_+$.
\par 
Fix $h\in\R$ and note that since $\partial_x\tilde f_+=0$ on $\Sigma_+$, the function $x\mapsto \tilde f_+(x,h)$ is constant on each component $E_k(h)$ of $A(h)$. Denoting this constant by $F_k(h)$, we obtain the global representation
\begin{equation}\label{eq:global_rep_upper}
\tilde f_+(x,h)=\sum_{k\in I(h)} \mathbf 1_{E_k(h)}(x)\,F_k(h),
\qquad (x,h)\in\Sigma_+.
\end{equation}
Finally, for $v>0$ we have $(x,h)=T_+(x,v)$ with $h=H(x,v)$ and thus
\[f(x,v)=\tilde f_+(x,h)=F_{K(x,h)}(h)=F_{K(x,H(x,v))}\!\bigl(H(x,v)\bigr).\] 
An analogous argument on $\Omega_-:=\R\times(-\infty,0)$  yields functions $\{G_k\}_{k\in\N}$ such that
\[
f(x,v)=G_{K(x,H(x,v))}\!\bigl(H(x,v)\bigr)\qquad\text{for }v<0.
\]
This proves~\ref{it:rep_components} which leads to~\ref{it:rho_components} upon integration in $v$.

Write $f=f_+\mathbf 1_{\{v>0\}}+f_-\mathbf 1_{\{v<0\}}$. Since $f_\pm$ are $C^1$ on their respective half-spaces,
the distributional derivative $\partial_v f$ contains a possible Dirac mass on $\{v=0\}$ with coefficient
$f(x,0^+)-f(x,0^-)$.
Inserting this into the distributional form of~\eqref{eq:stat_vlasov} shows that this coefficient must vanish after
multiplication by $\partial_x\Phi(\rho(x))$, i.e.
\begin{equation}\label{eq:jump_coeff}
\partial_x\Phi(\rho(x))\bigl(f(x,0^+)-f(x,0^-)\bigr)=0
\qquad\text{(in the sense of distributions in $x$)}.
\end{equation}
Using~\ref{it:rep_components}, for $v>0$ we have
$f(x,v)=F_{K(x,H(x,v))}(H(x,v))$ and $H(x,v)\to \Phi(\rho(x))$ as $v\to0^+$, hence
\[
\partial_x\Phi(\rho(x))\,f(x,0^+)
=\lim_{v\to0^+}\partial_x\Phi(\rho(x))\,F_{K(x,H(x,v))}(H(x,v)).
\]
Similarly, for $v<0$,
\[
\partial_x\Phi(\rho(x))\,f(x,0^-)
=\lim_{v\to0^-}\partial_x\Phi(\rho(x))\,G_{K(x,H(x,v))}(H(x,v)).
\]
Therefore \eqref{eq:jump_coeff} is equivalent to~\ref{it:jump_condition_components}.

\par 
Conversely, \ref{it:rep_components} implies the stationary equation holds classically on $\{v\neq0\}$, and
\ref{it:jump_condition_components} eliminates the boundary contribution at $v=0$, yielding the weak formulation.
\end{proof}

\begin{remark}\label{rem:component_gluing}
As a consistency check, the representation in Proposition~\ref{prop:Fk_Gk} is naturally organized
along the nested connected components of the sublevel sets $A(h)$. Since $A(h_1)\subseteq A(h_2)$
for $h_1<h_2$, these components may split or merge as the level $h$ increases. For the
families $\{F_k\}$ and $\{G_k\}$ arising from a given function $f\in C^1(\{v>0\}\cup\{v<0\})$,
the corresponding compatibility at merger levels is automatic, since these profiles are obtained
by restricting the single $C^1$ functions $\tilde f_+$ and $\tilde f_-$ on $\Sigma_+$ and $\Sigma_-$. 
If one were instead to start from arbitrary componentwise profiles and attempt to reconstruct $f$
through~\eqref{eq:Fk_Gk}, then matching of the values at merger levels would be the natural
continuity requirement, and matching of the first derivatives with respect to $h$ the natural
$C^1$-compatibility requirement. We do not formalize this reconstruction issue here.

By contrast, condition~\ref{it:jump_condition_components} only expresses the compatibility across
$v=0$ required by the weak formulation; it is weaker than full $C^1$-matching across $\{v=0\}$.
\end{remark}

\begin{remark}\label{rem:local_F_of_H}
The representation~\eqref{eq:global_rep_upper} is a global version of the local fact that for every
$(x_0,v_0)$ with $v_0>0$ there exists a neighborhood $U(x_0,v_0)\subseteq\R^2$ and a function $F_{U(x_0,v_0)}$ such that
\begin{equation} \label{eq:implicit_local}
f(x,v)=F_{U(x_0,v_0)}(H(x,v)) \quad  \text{for} \ (x,v)\in U(x_0,v_0).
\end{equation}
This follows from the implicit function theorem since $\nabla H(x,v)\neq 0$ on $\{v>0\}$.
\end{remark}

The previous proposition simplifies substantially when we restrict to stationary states which are real-analytic
on the upper and lower half phase spaces. Arguing on the upper half space, given $(x_0,v_0)\in\R\times(0,\infty)$,
the local representation~\eqref{eq:implicit_local} yields a locally defined real-analytic function of the energy,
denoted $F_{U(x_0,v_0)}$, such that $f=F_{U(x_0,v_0)}(H)$ near $(x_0,v_0)$.
Given any $(x_1,v_1)\in\R\times(0,\infty)$, we connect it to $(x_0,v_0)$ by a smooth curve in $\R\times(0,\infty)$ and,
by analytic continuation, we obtain a ``universal''
real-analytic function $F_1=F_1(h)$ on the energy range $(m,\infty)$, where
\[
m:=\inf_{x\in\R}\Phi(\rho(x)),
\]
such that $f(x,v)=F_1(H(x,v))$ for all $v>0$. The same argument on $\R\times(-\infty,0)$ yields a real-analytic function $F_2$ on $(m,\infty)$ such that $f(x,v)=F_2(H(x,v))$ for all $v<0$. \par We obtain:

\begin{proposition}\label{prop:stationary_analytic_F1F2}
Let $f=f(x,v)$ be real-analytic on $\{v < 0\} \cup \{v > 0\}$ and with $\rho = \int f\,\rd v$ real-analytic on $\R$. Then $f$ is a weak stationary state of~\eqref{vlasov0} if and only if there exist real-analytic functions $F_1,F_2$ such that
\begin{equation}\label{eq:F1F2_analytic}
f(x,v)=
\begin{dcases}
F_1\bigl(H(x,v)\bigr), & v>0,\\
F_2\bigl(H(x,v)\bigr), & v<0.
\end{dcases}
\end{equation}
Moreover,
\begin{equation}\label{eq:rho_F1F2}
\rho(x)=\int_0^\infty 
F_1\Bigl(\frac{v^2}{2}+\Phi(\rho(x))\Bigr)+
F_2\Bigl(\frac{v^2}{2}+\Phi(\rho(x))\Bigr)
\,\rd v,
\end{equation}
and for every $z\in \R$ which is not a critical value of $\Phi(\rho)$ one has $F_1(z)=F_2(z)$.
\end{proposition}

Note that the right-hand side of~\eqref{eq:rho_F1F2} depends on $x$ only through the value $\rho(x)$.
Hence, for every $x\in\R$, $\rho(x)$ satisfies the fixed-point equation
\begin{equation}\label{eq:algebraic_rho}
\rho = \int_0^\infty F\Big(\frac{v^2}{2} + \Phi(\rho)\Big) \, \rd v,
\qquad F\coloneqq F_1+F_2.
\end{equation}
Any isolated solution $\rho$ of~\eqref{eq:algebraic_rho} yields a spatially homogeneous stationary profile of~\eqref{eq:stat_vlasov} through~\eqref{eq:F1F2_analytic}, and additionally, if $F_1=F_2$ then the flux $J=0$ and we obtain an equilibrium.
\par 
The next result provides sufficient conditions for the solvability of~\eqref{eq:algebraic_rho} for nonnegative stationary states. The proof is a simple exercise in real analysis and is therefore omitted.  

\begin{proposition}\label{prop:algebraic_unique_rho}
Assume that $F:\R\to[0,\infty)$ is nonincreasing and satisfies
\[
\int_1^\infty \frac{F(z)}{\sqrt{z}}\,\rd z<\infty, \qquad \int_0^\infty F\Big(\frac{v^2}{2}+\Phi(0)\Big)\,\rd v > 0,
\]
and that $\Phi:[0,\infty)\to\R$ is continuous and nondecreasing with $\Phi(\rho)\to+\infty$ as $\rho\to+\infty$.
Then the map
\begin{equation} \label{eq:map_Prho}
P(\rho)\coloneqq \int_0^\infty F\Big(\frac{v^2}{2}+\Phi(\rho)\Big)\,\rd v
\end{equation}
is finite for every $\rho\ge 0$, continuous, and nonincreasing on $[0,\infty)$, with $P(\rho)\to 0$ as $\rho\to\infty$.
Consequently, the fixed-point equation $\rho=P(\rho)$ admits a unique solution $\rho >  0$.
\end{proposition}

Combining Proposition~\ref{prop:stationary_analytic_F1F2} with Proposition~\ref{prop:algebraic_unique_rho}, we infer that, under the above monotonicity assumptions on $F$ and $\Phi$, any nonnegative weak stationary state in this regularity class is necessarily spatially homogeneous. The assumption that $\Phi$ is nondecreasing is natural in the hydrodynamical case $\Phi=h'$ with $h''\ge 0$. \par 
In this case, perturbations reduce to perturbations of spatially homogeneous steady states $f^\ast(v)$, and moreover, in view of the weighted formulation from Section~\ref{section_main}, one may write
$f=f^\ast+\omega u$ (with $\omega$ as in~\eqref{weight}) and set $g^\ast\coloneqq f^\ast/\omega$. Repeating the Banach fixed-point argument for the perturbation $u=g-g^\ast$ yields local analytic well-posedness for small analytic perturbations of such homogeneous steady states,
with constants depending additionally on suitable analytic norms of~$g^\ast$.

We next give sufficient conditions on $\Phi$ and $F$ under which the function $\rho \mapsto P(\rho)$ in~\eqref{eq:map_Prho} is real analytic on an open interval of $\R$.

\begin{proposition}
Assume the following.
\begin{enumerate}[(i)]
\item $\Phi$ is real analytic on an open interval $I\subseteq \R$.

\item There exist an open interval $J\subseteq \R$ and $\eps>0$ such that
\[
\big[\inf_{\rho \in I} \Phi(\rho), \infty\big) \subseteq J,
\]
and $F$ is real-analytic on $J$ and admits a holomorphic extension, still denoted by $F$, to the strip
\[
\Omega:=J+i(-\eps,\eps) \subseteq \C.\]

\item There exist $C>0$ and $\beta>1$ such that
\[
|F(z)| \le \frac{C}{1+|z|^{\beta/2}}
\qquad \text{for all } z\in \Omega.
\]
\end{enumerate}

Then the function $P(\rho)$ defined in~\eqref{eq:map_Prho} is real-analytic on $I$.
\end{proposition}

\begin{proof}
For $\alpha\in\Omega$, define
\begin{equation} \label{eq:map_S}
S(\alpha):=\int_0^\infty F\Bigl(\frac{v^2}{2}+\alpha\Bigr)\,\rd v.
\end{equation}
We first check that $S$ is well defined on $\Omega$. Let $\alpha\in\Omega$, so
$\alpha=a+ib$ with $a\in J$ and $|b|<\eps$. Since
\[
\big[\inf_{\rho\in I}\Phi(\rho),\infty\big)\subseteq J
\]
and $J$ is an interval, it follows that
\[
a+\frac{v^2}{2}\in J \qquad \text{for every } v\ge 0.
\]
Hence
\[
\frac{v^2}{2}+\alpha\in \Omega \qquad \text{for all } \alpha\in\Omega,\ v\ge 0,
\]
so the integrand is well defined.

Now let $K$ be a compact subset of $\Omega$. Since $K$ is bounded, there exists $M_K>0$ such that
$|\alpha|\le M_K$ for all $\alpha\in K$. By assumption (iii),
\[
\Bigl|F\Bigl(\frac{v^2}{2}+\alpha\Bigr)\Bigr|
\le \frac{C}{1+\bigl|\frac{v^2}{2}+\alpha\bigr|^{\beta/2}}
\le \frac{C_K}{1+v^\beta},
\qquad \alpha\in K,\ v\ge 0,
\]
for some constant $C_K>0$. Since $\beta>1$, the right-hand side is integrable on $(0,\infty)$.
Therefore $S$ is well defined and continuous on $\Omega$ by dominated convergence.

Let $\gamma$ be a closed piecewise $C^1$ curve in $\Omega$. Since the image of $\gamma$ is compact in $\Omega$,
the same estimate gives an integrable bound along $\gamma$, and Fubini's theorem yields
\[
\int_\gamma S(\alpha)\,\rd\alpha
=
\int_0^\infty \left(\int_\gamma F\Bigl(\frac{v^2}{2}+\alpha\Bigr)\,\rd\alpha\right)\rd v.
\]
For each fixed $v\ge 0$, the map
\[
\alpha\mapsto F\Bigl(\frac{v^2}{2}+\alpha\Bigr)
\]
is holomorphic on $\Omega$, since $F$ is holomorphic on $\Omega$. By Cauchy's theorem,
\[
\int_\gamma F\Bigl(\frac{v^2}{2}+\alpha\Bigr)\,\rd\alpha=0.
\]
Hence
\[
\int_\gamma S(\alpha)\,\rd\alpha=0.
\]
Since $S$ is continuous on $\Omega$, Morera's theorem implies that $S$ is holomorphic on $\Omega$.
In particular, its restriction to $J$ is real analytic. Finally, since both $S|_J$ and $\Phi$ are real analytic, it follows that $P=S\circ\Phi$ is real analytic on $I$.
\end{proof}

As a consequence, the function
\[
R(\rho):=\rho-P(\rho)
\]
is real analytic on $I$. Hence either $R$ vanishes identically on $I$, or its zeros in $I$ are isolated. In particular, every fixed point $\rho_\ast\in I$ of
\[
\rho=P(\rho)
\]
such that
\[
1-P'(\rho_\ast)\neq 0
\]
is isolated. Therefore, by the discussion following~\eqref{eq:algebraic_rho}, such a fixed point yields a spatially homogeneous stationary profile.

The holomorphic extension required above is closely related to Paley--Wiener-type results.
For instance, exponential decay of the Fourier transform $\widehat F$ of $F$ yields a holomorphic extension of $F$ to a strip in the complex plane centered on the real axis,
while additional derivative bounds on $\widehat F$ yield decay of the extension in that strip. For these Fourier-analytic criteria, see, for example,~\cite{stein2010complex}.

We conclude with a result in the spirit of the above discussion. For $f\in L^1(\R)$, we define the Fourier transform
\[
\widehat f(\xi):=\bigl(\mathcal F_{\alpha\to\xi}f\bigr)(\xi)
=\int_{-\infty}^{\infty} f(\alpha)e^{-2\pi i\alpha\xi}\,\rd\alpha,
\qquad \xi\in\R,
\]
and the inverse Fourier transform
\[
\check g(\alpha):=\bigl(\mathcal F^{-1}_{\xi\to\alpha}g\bigr)(\alpha)
=\int_{-\infty}^{\infty} g(\xi)e^{2\pi i\alpha\xi}\,\rd\xi,
\qquad \alpha\in\R.
\]

\begin{proposition}\label{prop:S_fourier_analytic}
Let $F\in L^1(\R)$ and assume that its Fourier transform satisfies the decay estimate
\[
|\widehat F(\xi)|\le M e^{-2\pi\delta|\xi|},
\qquad \xi\in\R,
\]
for some fixed constants $M>0$ and $\delta>0$. Then the function $S$ defined in~\eqref{eq:map_S} admits a holomorphic extension to the strip
\[
\R+i\Bigl(-\frac{\delta}{2},\frac{\delta}{2}\Bigr)\subseteq\C.
\]
In particular, $S$ is real analytic on $\R$.
\end{proposition}

\begin{proof}
Since $F\in L^1(\R)$, its Fourier transform $\widehat F$ exists as a bounded uniformly continuous function on $\R$.
The decay estimate on $\widehat F$ implies that $\widehat F\in L^1(\R)\cap L^2(\R)$. Hence the inversion
formula yields a continuous representative of $F$, still denoted by $F$, and
\[
F=\check{\widehat F}.
\]

Note that the alternative representation
\[
S=F*h,\qquad h(\alpha):=\frac{1}{\sqrt{2|\alpha|}}\,\mathbf 1_{\{\alpha<0\}},
\]
holds for the function $S$. Moreover,
\[
h\in L^1_{\mathrm{loc}}(\R)\cap L^{2,\infty}(\R),
\]
where $L^{2,\infty}(\R)$ denotes the weak $L^2$ space. Since $\widehat F\in L^2(\R)$, Plancherel's theorem
implies that $F\in L^2(\R)$, and therefore $F\in L^r(\R)$ for every $1\le r\le 2$. By the weak Young
inequality, it follows that
\[
S\in L^q(\R)\qquad\text{for every }2<q<\infty.
\]

A direct computation gives
\[
\widehat h(\xi)=\frac{1+i\,\operatorname{sgn}(\xi)}{\sqrt{8|\xi|}},
\qquad \xi\in\R\setminus\{0\},
\]
so that we find
\[
\widehat S(\xi)
=
\widehat F(\xi)\,\frac{1+i\,\operatorname{sgn}(\xi)}{\sqrt{8|\xi|}}.
\]
Then we have by the inversion formula
\begin{equation}\label{eq:S_fourier_rep}
S(\alpha)=\frac{1}{\sqrt8}\int_{-\infty}^{\infty}
\widehat F(\xi)\,\frac{1+i\,\operatorname{sgn}(\xi)}{\sqrt{|\xi|}}\,e^{2\pi i\alpha\xi}\,\rd\xi.
\end{equation}

We now define, for $\alpha\in \R+i\bigl(-\frac{\delta}{2},\frac{\delta}{2}\bigr)$,
\[
\widetilde S(\alpha):=
\frac{1}{\sqrt8}\int_{-\infty}^{\infty}
\widehat F(\xi)\,\frac{1+i\,\operatorname{sgn}(\xi)}{\sqrt{|\xi|}}\,e^{2\pi i\alpha\xi}\,\rd\xi.
\]
We claim that this integral is absolutely convergent. Indeed,
\[
\frac{1}{\sqrt8}\left|
\widehat F(\xi)\,\frac{1+i\,\operatorname{sgn}(\xi)}{\sqrt{|\xi|}}\,e^{2\pi i\alpha\xi}
\right|
\le
\frac12\,\frac{|\widehat F(\xi)|}{\sqrt{|\xi|}}\,e^{2\pi |\xi|\,|\Im\alpha|}
\le
\frac{M}{2}\,\frac{e^{-\pi\delta |\xi|}}{\sqrt{|\xi|}},
\]
and the right-hand side belongs to $L^1(\R)$. Hence $\widetilde S$ is well defined and continuous on the strip by dominated convergence. Moreover, for real $\alpha$,
the preceding identity shows that $\widetilde S(\alpha)=S(\alpha)$.

It remains to prove that $\widetilde S$ is holomorphic in the strip. Let $\gamma$ be a closed piecewise $C^1$
curve in $\R+i\bigl(-\frac{\delta}{2},\frac{\delta}{2}\bigr)$. Since the image of $\gamma$ is compact,
the same estimate shows that the integrand belongs to $L^1(\R\times \gamma;\,\rd\xi\,|\rd\alpha|)$, and Fubini's theorem yields
\[
\int_\gamma \widetilde S(\alpha)\,\rd\alpha
=
\frac{1}{\sqrt8}\int_{-\infty}^{\infty}
\widehat F(\xi)\,\frac{1+i\,\operatorname{sgn}(\xi)}{\sqrt{|\xi|}}
\left(\int_\gamma e^{2\pi i\alpha\xi}\,\rd\alpha\right)\rd\xi.
\]
For each fixed $\xi\in\R$, the map $\alpha\mapsto e^{2\pi i\alpha\xi}$ is entire, and therefore
\[
\int_\gamma e^{2\pi i\alpha\xi}\,\rd\alpha=0.
\]
Hence
\[
\int_\gamma \widetilde S(\alpha)\,\rd\alpha=0.
\]
Since $\widetilde S$ is continuous in the strip, Morera's theorem implies that $\widetilde S$ is holomorphic there.
Because $\widetilde S$ extends $S$, the proof is complete.
\end{proof}

\appendix

\section{Banach spaces of analytic functions} \label{section_complete_metric_space}
\begin{lemma}
Let $\lambda > 0$ and consider the space $W_\lambda$ of those functions $f \in C^\infty(\mathbb{R})$ satisfying \[\| f \|_\lambda = \sum_{k \geq 0} \frac{\lambda^k}{k!} \| 
\partial_x^k f\|_\infty < \infty.\] Then $W_\lambda$ is a Banach space under the norm $\| \idot \|_\lambda.$
\end{lemma}
\begin{proof}
Let $(f_n) \subseteq W_\lambda$ be a Cauchy sequence. Then, for each $k \in \mathbb{N}_0$, $(\partial_x^k f_n)$ is a Cauchy sequence of the Banach space $C_b(\mathbb{R})$ of bounded continuous functions, and so, there exists $g_k \in C_b(\mathbb{R})$ such that $\|\partial_x^k f_n - g_k \|_\infty \to 0$ as $n \to \infty$. Set $f = g_0$. We claim that for each $k \in \mathbb{N}_0$, $g_k = \partial_x^k f$. Let $\varphi \in C_c^\infty(\mathbb{R})$. Then
\[\int_\mathbb{R} \partial_x f_n \cdot \varphi \, \rd x = - \int_\mathbb{R} f_n \cdot \partial_x \varphi \, \rd x \]
which, by letting $n \to \infty$ yields
\[ \int_\mathbb{R} g_1 \cdot \varphi \, \rd x = - \int_\mathbb{R} f \cdot \partial_x \varphi \, \rd x.\]
Hence, $g_1$ is the weak derivative of $f$, and since $g_1$ is continuous, we have that $f \in C^1(\mathbb{R})$ and $g_1 = \partial_x f$. The claim follows by a standard inductive argument. \par 
Moreover, given $\varepsilon > 0$, there exists $N \in \mathbb{N}$ such that for $n,m \geq N$, $\|f_n - f_m \|_\lambda < \varepsilon.$ Letting $m \to \infty$ combined with Fatou's lemma results in $\|f_n - f \|_\lambda < \varepsilon$ for $n \geq N$. Thus, $\|f\|_\lambda \leq \|f_N - f \|_\lambda + \|f_N \|_\lambda < \infty $. This proves that $(f_n)$ has a limit in $W_\lambda$, establishing the result.
\end{proof}

\begin{lemma}
Let $\lambda,T > 0$ and consider the space $Y_{\lambda,T}$ of those functions $f \in C\big([0,T]; C^\infty(\mathbb{R})\big)$ satisfying \[\sup\limits_{t\in[0,T]} \|f(t) \|_\lambda = \sup\limits_{t\in[0,T]} \sum_{k \geq 0} \frac{\lambda^k}{k!} \| 
\partial_x^k f(t) \|_\infty < \infty.\] Then $Y_{\lambda,T}$ is a Banach space under the norm $\sup\limits_{[0,T]} \| \idot \|_\lambda$.
\end{lemma}
\begin{proof}
Let $(f_n) \subseteq Y_{\lambda,T}$ be a Cauchy sequence. Then, for each $t \in [0,T]$, $(f_n(t))$ is Cauchy in the space $W_\lambda$ defined in the previous lemma. Therefore, there exists $g_t \in W_\lambda$ such that $\|f_n(t) - g_t \|_\lambda \to 0$ as $n \to \infty$. Define $f : [0,T] \to C^\infty(\mathbb{R})$ by $f(t) = g_t$. Note that, for every $n,m \in \mathbb{N}$ and any $t \in [0,T]$, there holds
\[\big| \|f(t) - f_n(t) \|_\lambda - \| f_m(t) - f_n(t) \|_\lambda   \big| \leq \| f(t) - f_m(t) \|_\lambda \]
and so
\[ \| f(t) - f_n(t) \|_\lambda = \lim\limits_{m \to \infty}\| f_m(t) - f_n(t) \|_\lambda. \]\par 
Given $\varepsilon > 0$, there exists $N \in \mathbb{N}$ such that for $n,m \geq N$ and all $t \in [0,T]$ we have $\|f_n(t) - f_m(t) \|_\lambda \leq \varepsilon$. Letting $m \to \infty$ yields $\|f_n(t) - f(t) \|_\lambda < \varepsilon$ for all $n \geq N$ and $t \in [0,T]$. A standard argument then implies that $f$ is continuous in time, which concludes the proof.
\end{proof}
\begin{remark}
The previous two lemmas also hold if $C^\infty(\mathbb{R})$ is replaced by $C^\infty(\mathbb{R}^2)$, in which the norm $\| \idot \|_\lambda$ becomes \[\| f \|_\lambda = \sum_{k,\ell \geq 0} \frac{\lambda^{k+\ell}}{k! \ell!} \| 
\partial_x^k \partial_v^\ell f\|_\infty. \]
\end{remark}

\begin{proposition}\label{prop:Z_complete_time}
Let $\lambda_0, K, T > 0 $ be such that $\lambda_0 > (K+1)T$ and consider the function $\lambda(t) = \lambda_0 - (K+1)t$ for $t \in [0,T]$. The space $\mathrm{Z}_{\lambda_0, K, T}$ defined in (\ref{spaceZ}) is a Banach space under the norm $\|\idot\|_\Z$ given by \eqref{normZ}.
\end{proposition}

\begin{proof}
Set $\lambda_T=\lambda(T)$ and let $(f_n)\subseteq \mathrm Z_{\lambda_0,K,T}$ be Cauchy in $\|\idot\|_{\Z}$.
Since $\lambda(t)\ge \lambda_T$ for all $t$, we have
\[
\| \idot \|_{\lambda_T}\le \|\idot\|_{\lambda(t)},\qquad |\idot |_{\lambda_T,1}\le |\idot|_{\lambda(t),1}.
\]
Hence $(f_n)$ is Cauchy in space $Y_{\lambda_T,T}$ (in $\R^2$), so there exists
$f\in C\big([0,T];C^\infty(\R^2)\big)$ such that
\[
\sup_{t\in[0,T]}\|f_n(t)-f(t)\|_{\lambda_T}\to 0.
\]

Fix $t\in[0,T]$. Since
\[
\|f_n(t)-f_m(t)\|_{\lambda(t)}\le \sup_{\tau\in[0,T]}\|f_n(\tau)-f_m(\tau)\|_{\lambda(\tau)}
\le \|f_n-f_m\|_{\Z},
\]
the sequence $(f_n(t))$ is Cauchy in the Banach space $W_{\lambda(t)}$, hence converges in $\|\idot\|_{\lambda(t)}$
to a limit, which must equal $f(t)$ by uniqueness of the limit in the weaker norm $\|\idot\|_{\lambda_T}$.
Therefore,
\[
\|f_n(t)-f(t)\|_{\lambda(t)}\to 0\qquad\text{for every }t\in[0,T].
\]
Taking the supremum over $t \in [0,T]$ and using $\limsup$,
\begin{align*}
\sup_{t\in[0,T]}\|f_n(t)-f(t)\|_{\lambda(t)}
& \le \limsup_{m\to\infty}\sup_{t\in[0,T]}\|f_n(t)-f_m(t)\|_{\lambda(t)} \\
& \le \limsup_{m\to\infty}\|f_n-f_m\|_{\Z}\to 0 \quad \text{as} \ n \to \infty.
\end{align*}
For the integral term, fix $n$ and use Fatou's lemma to obtain
\[
\int_0^T |f_n(t)-f(t)|_{\lambda(t),1}\,\rd t
\le \liminf_{m\to\infty}\int_0^T |f_n(t)-f_m(t)|_{\lambda(t),1}\,\rd t.
\]
Since $(f_n)$ is Cauchy in $\|\idot\|_{\Z}$, the right-hand side tends to $0$ as $n\to\infty$.
Hence $\|f_n-f\|_{\Z}\to 0$, proving completeness.
\end{proof}

\section{Maximum principle} \label{section_maximum_principle}

Suppose that $g = g(t,x,v)$ is a smooth solution of the equation 
\[\partial_t g + v \cdot \partial_x g + a(t,x) \cdot \partial_v g = G(t,x,v) \]
where $a$ is smooth in time and globally Lipschitz in space, and $G$ is smooth and bounded. 
\par Then 
\begin{equation} \label{maximumprinciple}
\frac{\rd}{\rd t} \|g(t) \|_{\infty} \leq \| G(t)\|_{\infty}.
\end{equation} Indeed, using characteristics we have:
\[\frac{\rd}{\rd t} g(t, x(t), v(t)) = G(t, x(t), v(t)), \]
where $(x(t), v(t))$ solves
\begin{equation*}
\begin{cases}
\dot{x}(t) = v(t), \\
\dot{v}(t) = a(t,x(t)), \\
x(0) = x_0, \ v(0) = v_0.
\end{cases}
\end{equation*}
Fix $t \in (0,T)$ and $h>0$ such that $t+h \leq T$. Integrating over $[t, t + h]$ yields
\begin{align*}
g(t+h, x(t+h), v(t+h)) & = g(t, x(t), v(t)) + \int_t^{t+h} G(\tau, x(\tau), v(\tau)) \, \rd \tau 
\end{align*}
hence
\begin{align*}
\big|g(t+h, x(t+h), v(t+h))\big|
& \leq \|g(t)\|_\infty + \int_t^{t+h} \| G(\tau)\|_{\infty} \, \rd \tau.
\end{align*}
Since the initial values $(x_0,v_0)$ are arbitrary, it follows that
\begin{equation*}
\|g(t+h) \|_\infty - \|g(t) \|_\infty \leq \int_t^{t+h} \| G(\tau)\|_{\infty} \, \rd \tau.
\end{equation*}
Similarly, for $h>0$ such that $t-h \geq 0$,
\begin{equation*}
\|g(t) \|_\infty - \|g(t-h) \|_\infty \leq \int_{t-h}^t \| G(\tau)\|_{\infty} \, \rd \tau.
\end{equation*}
The inequality (\ref{maximumprinciple}) is obtained by dividing both previous expressions by $h$ and then letting $h \to 0^+$.

\section*{Acknowledgments}
We thank Anne Nouri and Pierre-Emmanuel Jabin for helpful comments on an early version of this manuscript. We also thank the anonymous referees for suggestions that substantially improved the final presentation. This work was supported by the Austrian Science Fund (FWF), project 10.55776/F65, and by KAUST baseline funding.

\end{document}